\documentclass[aap]{imsart}
\usepackage{enumitem}
\usepackage{diagbox} 
\usepackage{bm}
\RequirePackage{amsthm,amsmath,amsfonts,amssymb}
\RequirePackage[numbers,sort&compress]{natbib}
\RequirePackage[colorlinks,citecolor=blue,urlcolor=blue]{hyperref}
\RequirePackage{graphicx}

\startlocaldefs
\theoremstyle{plain}

\newtheorem{theorem}{Theorem}[section]
\newtheorem{lemma}[theorem]{Lemma}
\theoremstyle{remark}
\newtheorem{corollary}{Corollary}
\newtheorem{remark}{Remark}
\newtheorem{definition}[theorem]{Definition}
\newtheorem{proposition}[theorem]{Proposition}

\endlocaldefs

\begin{document}

\begin{frontmatter}
\title{A Malliavin Calculus Approach to Backward Stochastic Volterra Integral Equations}
\runtitle{A Malliavin Calculus Approach to BSVIE}

\begin{aug}
\author[A]{\fnms{Qian}~\snm{Lei}\ead[label=e1]{qian.lei@ntu.edu.sg}\orcid{https://orcid.org/0009-0006-1569-3542}}
\and
\author[B]{\fnms{Chi Seng}~\snm{Pun}\ead[label=e2]{cspun@ntu.edu.sg}\orcid{https://orcid.org/0000-0002-7478-6961}}
\address[A]{School of Physical and Mathematical Sciences, Nanyang Technological University, Singapore \printead[presep={ ,\ }]{e1}}

\address[B]{School of Physical and Mathematical Sciences, Nanyang Technological University, Singapore \printead[presep={,\ }]{e2}}
\end{aug}

\begin{abstract}
In this paper, we establish existence, uniqueness, and regularity properties of the solutions to multi-dimensional backward stochastic Volterra integral equations (BSVIEs), whose (possibly random) generator reflects nonlinear dependence on both the solution process and the martingale integrand component of the adapted solutions, as well as their diagonal processes. The well-posedness results are developed with the use of Malliavin calculus, which renders a novel perspective in tackling with the challenging diagonal processes while contrasts with the existing methods. We also provide a probabilistic interpretation of the classical solutions to the counterpart semi-linear partial differential equations through the explicit adapted solutions of BSVIEs. Moreover, we formulate with BSVIEs to explicitly characterize dynamically optimal mean-variance portfolios for various stochastic investment opportunities, with the myopic investment and intertemporal hedging demands being identified as two diagonal processes of BSVIE solutions.
\end{abstract}



\end{frontmatter}

\section{Introduction} \label{sec:intro}
In this paper, we focus on a fixed finite time horizon $[0,T]$ and introduce the notation of $\Delta[a,b]:=\{(t,s):a\leq t\leq s\leq b\}$. The backward stochastic Volterra integral equations (BSVIEs) under our consideration takes the form
\begin{equation} \label{General BSVIE}
    \begin{cases}
        \begin{aligned}
            dY(t,s) & ~=~  g(t,s,Y(t,s),Z(t,s),Y(s,s),Z(s,s))ds+Z(t,s)dB(s), \\
            
            Y(t,T) & ~=~ \xi(t), \quad 0\leq t \leq s\leq T,
        \end{aligned} 
    \end{cases}
\end{equation} 
where $\{B(s)\}_{0\leq s\leq T}$ is a standard $n$-dimensional Brownian motion defined on some complete filtered probability space $(\Omega,\mathcal{F},P,\{\mathcal{F}_s\}_{0\leq s\leq T})$. Let $\mathbb{F}:=\{\mathcal{F}_s\}_{0\leq s\leq T}$ be the filtration generated by the Brownian motion and the $P$-null sets, and $\mathcal{F}=\mathcal{F}_T$. Moreover, the generator $g:\Delta[0,T]\times\Omega\times\mathbb{R}^{2k}\times\mathbb{R}^{2(k\times n)} \to \mathbb{R}^k$ and the $\mathcal{F}_T$-measurable terminal datum \(\xi:[0,T]\times\Omega\to\mathbb{R}^k\) are both given and they could be random. We have suppressed the argument $\omega\in\Omega$ in \eqref{General BSVIE} and the convention remains throughout this paper as it is more of an indication of involving randomness. Our goal is to find a pair of $\mathbb{F}$-adapted processes $Y:\Delta[0,T]\times\Omega\to\mathbb{R}^k$ and $Z:\Delta[0,T]\times\Omega\to\mathbb{R}^{k\times n}$ satisfying the equation \eqref{General BSVIE} above. 

The BSVIEs \eqref{General BSVIE} are distinguished from classical backward stochastic differential equations (BSDEs) studied in \cite{Pardoux1990,Karoui1997} in two main aspects: (1) the stochastic system \eqref{General BSVIE} involves two temporal variables $(t,s)$ with an ordering $t\le s$, where $s$ acts as a dynamic variable and $t$ acts as a reference time point. Moreover, both the generator $g$ and the terminal datum $\xi$ are varying in time $t$.; (2) the generator $g$ depends on the values of the solution process $Y$ and the martingale integrand process $Z$ at not only $(t,s)$ but also $(s,s)$. We refer to $\{(Y, Z)(s,s)\}_{0\leq s\leq T}$ as a pair of \textit{diagonal} processes of the processes $\{(Y,Z)(t,s)\}_{0\leq t\leq s\leq T}$. It is clear that when the generator is independent of these diagonal terms, the BSVIE \eqref{General BSVIE} reduces to a well-posed family of BSDEs parameterized by $t\in [0, T]$, which has been extensively studied in \cite{Karoui1997}. Thus, the key analytical challenges and significance of studying \eqref{General BSVIE} lie in their dependence on the diagonal processes.

The studies of \eqref{General BSVIE} have recently attracted much attention due to its potential to generalize and amplify the impacts of the well-studied BSDE theory. The BSVIEs found applications in various fields, especially where we intend to incorporate behavioral/psychological factors into the modeling. To name a few, we list them and the related works as follows:
\begin{enumerate}
    \item \textbf{Stochastic differential utility}, introduced in \cite{Duffie1992}, can be interpreted as the solution component of the adapted solution to a BSDE; see \cite{Karoui1997}. However, when we incorporate with the empirical findings (see \cite{Ekeland2006,Laibson1997,Strotz1955}) on human's preferences, which are time-varying or reference-dependent, we need a more general framework than BSDEs, i.e., BSVIEs, to describe the associated recursive utility; see \cite{Wei2017, Yan2019, Hamaguchi2021}.
    \item \textbf{Coherent risk measures}, introduced in \cite{artzner1999coherent} and further developed in \cite{cheridito2004coherent,Peng2004} for static or time-consistent (TC) settings, are also extendable to dynamic coherent risk measures with time-varying characteristics, calling for the use of BSVIEs as explored in \cite{yong2007continuous, wang2021recursive}.
    \item \textbf{Time-inconsistent (TIC) stochastic control problems}, where decision-makers' preferences and tastes rely on the initial-time (state) as a reference point, in contrast to conventional TC problems, Bellman's principle of optimality is no longer applicable. In this case, we have to determine which agent out of different dynamic optimality desires, normally among the three types discussed in \cite{Strotz1955}, namely pre-committed, sophisticated, and myopic agents. Sophisticated agents are common because of its tractability and time consistency in decision making, i.e, consistent planning in \cite{Pollak1968}. Many works have contributed to the characterization of the TC policy of the sophisticated agents with two main streams:
    \begin{itemize}[wide = 0pt, leftmargin = 1.3em]
    	\item[(I)] one leverages the spike variation approach to open- or closed-loop-type equilibrium controls (\cite{Hu2012,Hu2017,Huang2017,Han2021,Han2022}) so as to modify Pontryagin's maximum principle in \cite{peng1990general,Yong1999}. The corresponding first-order and second-order adjoint processes are characterized by BSVIEs instead of BSDEs. Moreover, the maximum condition is driven by the diagonal processes; see \cite{Yan2019, Hamaguchi2021} for more details.
    	\item[(II)] One interprets the TIC problems as intrapersonal games (whose players are indexed by time) and then the TC policy coincide with the subgame perfect equilibrium of the game; see \cite{Bjoerk2014, Bjoerk2017, Yong2012, Wei2017, Pun2018}, which deduce extended Hamilton-Jacobi-Bellman (HJB) system or equilibrium HJB equation (a fully nonlinear nonlocal partial differential equation (PDE)) to characterize the equilibrium solution. The Feynman-Kac-type formulas in \cite{Wang2021,lei2023nonlocal,LeiSDG,hernandez2023me,Hernandez2021a} show that the solution to the counterpart BSVIEs offers a probabilistic interpretation of these Nash equilibria. 
    \end{itemize}

    \item \textbf{Probabilistic representation of the solution to counterpart parabolic nonlocal PDEs} can be achieved with BSVIEs. Echoing to the previous point 3(II), we can extend the probabilistic interpretation to a broader class of nonlocal PDEs. Such a relationship is a key ingredient of developing a simulation-based deep learning algorithm for solving those nonlocal PDEs in high dimensions; see \cite{Han2018,Beck2019} for the inspiration of their studies on the local case. Hence, the insights gained from studying BSVIEs can benefit the studies on nonlocal PDE and its numerical or deep-learning solvers.
    
    
\end{enumerate}

In the existing literature on BSVIEs, most of their focuses have been on the dependence of the diagonal process \( Y(s, s) \) rather than \( Z(s, s) \) in the generator. However, for the practical significance of adopting BSVIEs, the inclusion of $Z(s,s)$ is crucial. To see this,
\begin{itemize}[wide = 0pt, leftmargin = 1.3em]
    \item[$-$] neglecting $Z(s,s)$ has long impeded the full characterization of open-loop equilibrium controls in certain TIC stochastic control problems. Indeed, \cite[Section 4.1]{yong2017linear} made a strong and hard-to-verify assumption: the almost sure continuity of the map $(t,s) \mapsto Z(t,s)$ was imposed as a sufficient condition for open-loop equilibrium control. Moreover, in \cite[Section 4]{Yan2019}, the characterization relies on a limiting procedure, hindering a fully local description. These issues are mitigated when \cite{Hamaguchi2021,hernandez2023me,Hernandez2021a} provides a rigorous approach to handle such a diagonal process.
    \item[$-$] the exclusion of the diagonal process \( Z(s, s) \) significantly limits the applicability of BSVIEs in analyzing the equilibrium solution. Technically speaking, the absence of $Z(s,s)$ restricts the form of equilibrium HJB PDEs, rendering the truncated formulations being too weak to effectively describe the equilibrium solutions. An example of which is when only the drift in the state process is controlled while the diffusion is not controllable. This degeneration highlights the strong relation between the concepts of nonlocality in BSVIEs (via $Y(s,s)$ and $Z(s,s)$) and consistent planning in TIC problems; see \cite{Wang2021}. 
\end{itemize}





Though it is significant to include the diagonal process $Z(s,s)$, its inclusion poses considerable challenges. They are mainly reflected on two main aspects: well-definedness and well-posedness. In what follows, we discuss about them and review the existing approaches.
\begin{enumerate}
    \item The well-definedness of the pair of  diagonal processes $\{(Y,Z)\}(s,s)\}_{0\leq s\leq T}$ is not clarified. A straightforward yet intuitive definition of $(Y,Z)(s,s)$ could be $(Y,Z)(s,\tau)|_{\tau=s}=\lim_{\tau\to s^+}(Y,Z)(s,\tau)$. It is legitimate for $Y(s,s)$ as $Y(s,\tau)$ is continuous in the $\tau$-direction for almost all trajectories. However, it is generally not well-defined for $Z(s,s)$ because the usual BSDE well-posedness theory only covers the integrability of $Z(s,\tau)$ rather than its pointwise behavior in $\tau$, leading the limiting process to be neither well-defined nor well-understood. A counterexample of which was demonstrated in \cite[Example 2.4]{Hamaguchi2021}, where there exists a deterministic process $Z(s, \tau)$ whose diagonal term $Z(s,s)$ is not well-defined. This challenge partially explains why most existing studies on BSVIEs focuse on only the well-defined $Y(s,s)$ and why the assumption of almost sure continuity of the map $(t,s)\mapsto Z(t,s)$ (desired in some proofs) is hard to verify. The existing literature presents two approaches to address the ill-definedness of $\{Z(s,s)\}_{0\leq s\leq T}$: 
    \begin{itemize}[wide = 0pt, leftmargin = 1.3em]
        \item[$-$] By enhancing $t$-directional regularity of $\{Z(t,s)\}_{0\leq t\leq s\leq T}$ to differentiability or absolute continuity, we can show the existence of a derivative (with respect to $t$) process $\{Z_t(t,s)\}_{0\leq t\leq s\leq T}$. From which, we can then interpret the diagonal process as $Z(s,s)=Z(t,s)+\int^s_tZ_t(\alpha,s)d\alpha$. By this approach, \cite{Hamaguchi2021,hernandez2023me,Hernandez2021a} explore the well-definedness and adaptedness of $\{Z(s,s)\}_{0\leq s\leq T}$, along with its various properties and the BSDE satisfied by $\{Z_t(t,s)\}_{0\leq t\leq s\leq T}$. 
        \item[$-$] By adopting a PDE-analytical approach for Markovian cases, \cite{Wang2021,lei2023nonlocal,LeiSDG} first establish the well-posedness of classical solutions for a class of PDEs with a dual-temporal-variable structure. It\^{o}'s lemma is then applied to show that the Markovian BSVIE admits an adapted solution of the form $(Y, Z)(t,s) = (u(t,s,X(s)),u_x(t,s,X(s))\sigma(s))$, where $ (u(t,s,x), X(s))$ are the unique solutions to the associated PDE and a forward SDE with diffusion \(\sigma(s)\). It is clear that \((Y,Z)(s,s)\) are well-defined in this context, as $(Y,Z)(t,s)$ inherits $s$-directional continuity from $(u(t,s,x), X(s),\sigma(s))$. Though, this approach is exclusive for the Markovian case.
    \end{itemize}
    \item The well-posedness of solutions to BSVIEs with $Z(s,s)$ is difficult to establish. In addition to the ill-definedness of \( Z(s,s) \), another challenge arises: when we attempt to adopt the methods in \cite{Hamaguchi2021, hernandez2023me, Hernandez2021a} that introduce regularity in the $t$-direction to ensure the existence of its derivative process $\{Z_t\}$ and interprets \( Z(s,s) = Z(t,s) + \int_t^s Z_t(\alpha, s) \, d\alpha \), we need to study a non-Lipschitz BSVIE system, jointly composed of the original BSVIE \eqref{General BSVIE} for \( (Y, Z)(t, s) \) and an induced BSVIE for their \( t \)-derivative processes \( (Y, Z)_t(t, s) \) given by
    \begin{equation} \label{SysYt}
    \begin{cases}
         \begin{aligned}
             dY_t(t,s) & ~=~ \big[g_t(t,s,Y(t,s),Z(t,s),Y(s,s),Z(s,s))+g_Y\cdot Y_t(t,s) \\
         & \qquad +g_Z\cdot Z_t(t,s)\big]ds +Z_t(t,s)dB(s), \\
        Y_t(t,T) & ~=~ \xi_t(t), \quad 0\leq t\leq s\leq T,
         \end{aligned}   
    \end{cases}
   \end{equation}
where \(\xi_t\) is the derivative of \(\xi\) with respect to \(t\) and \(g_t\), \(g_Y\), and \(g_Z\) are the partial derivatives of \(g\) with respect to \(t\), \(Y(t,s)\), and \(Z(t,s)\), respectively. All partial derivatives of $g$ in \eqref{SysYt} are evaluated at $(t,s,Y(t,s),Z(t,s),Y(s,s),Z(s,s))$. In what follows, we review the recent advances on studying the BSVIE system \eqref{General BSVIE}-\eqref{SysYt}.
\begin{itemize}[start=1, wide = 0pt, leftmargin = 1.3em]
    \item[$-$] To address the ill-definedness of $Z(s,s)$ and establishing the well-posedness of ${\bm Z(s,s)}$\textbf{-independent} BSVIE solutions, \cite{Hamaguchi2021} found that the BSVIE system \eqref{General BSVIE}-\eqref{SysYt} is decoupled as one can first solve the original BSVIE \eqref{General BSVIE} for \( (Y,Z) \) and all coefficients in the BSVIE \eqref{SysYt} for $(Y,Z)_t$ become known. With sufficient regularities, the \( (Y,Z)_t \)-BSVIE is linear and well-posed, parameterized by $t$, which ensures the processes \( (Y,Z)_t \) and all the diagonal processes (via integral expressions) to be well-defined.
    \item[$-$] \cite{Wang2022} has considered $Z(s,s)$ in their BSVIEs but they do not depend on $(Y,Z)(t,s)$, which simplified the well-posedness analysis. Specifically, in this case, we would have both $g_Y$ and $g_Z$ in \eqref{SysYt} equal to zero. The nonlinear dependence of the BSVIE generator $g$ on $(Y,Z)(t,s)$ is however crucial for modelling and generality. Hence, to clarify, when we stress on the difficulty of studying $Z(s,s)$ in BSVIEs, we attempt to preserve the most general  dependence structure of the generator.
    \item[$-$] 
    To the best of our knowledge, \cite{hernandez2023me, Hernandez2021a} were the only alternatives (to this paper) that attempts ${\bm Z(s,s)}$\textbf{-dependent BSVIEs}. However, their analytical framework requires a restrictive Lipschitz continuity assumption on \( g \) in \eqref{General BSVIE} as well as the generator in \eqref{SysYt}:
     \begin{equation} \label{nableg}
    \nabla g(t,s,(Y,Z)(t,s),(Y,Z)(s,s),(Y,Z)_t(t,s)) := g_t + g_Y \cdot Y_t(t,s) + g_Z \cdot Z_t(t,s) 
    \end{equation}
     with respect to \( (Y,Z)(t,s), (Y,Z)(s,s), (Y,Z)_t(t,s) \), uniformly over the other arguments. The Lipschitz continuity requirement on \( \nabla g \) appears to be impractical, since \( (Y,Z)_t(t,s) \) has yet to be determined and no prior estimates are available. Note that for the general case, \( g_Y \) and \( g_Z \) depend on both \( (Y, Z)(t, s) \) and \( (Y, Z)(s, s) \) and thus the Lipschitz coefficients of \( \nabla g \) with respect to \( (Y, Z)(t, s) \) and \( (Y, Z)(s, s) \) will depend on the unknown \( (Y, Z)_t(t, s) \). Hence, assuming that \( \nabla g \) is uniformly Lipschitz not only imposes conditions on $g$, but also peeks information about the undetermined \( (Y, Z)_t(t, s) \). This key assumption, going beyond the intrinsic parameters $(\xi,g)$, is difficult to verify. In practice, it appears that the only easily verifiable cases satisfying this condition are: (a) when \( g_Y \) and \( g_Z \) in \eqref{nableg} do not depend on \( (Y, Z)(t, s) \) or \( (Y, Z)(s, s) \) implying a linear generator \( g \) in \eqref{General BSVIE} on $(Y, Z)(t, s)$; (b) when \( (g,\xi) \) are independent of \(t\), implying from $(g_t,\xi_t)=(0,0)$ in \eqref{SysYt} that \((Y, Z)_t(t, s)=(0, 0) \), the \eqref{General BSVIE} in this case reduces to a trivial case, i.e., a family of BSDEs parameterized by $t$.
    \item[$-$] In the studies of \textbf{Markovian BSVIEs}, a PDE-based analytical approach is appealing; see \cite{Wang2021,lei2023nonlocal,LeiSDG}. In which, the conditions on the nonlinear homogeneous term in the associated PDEs are relaxed to allow the generator \( g \) of the BSVIE \eqref{General BSVIE} to exhibit nonlinear dependence on \((Y, Z)(t,s)\) and \((Y, Z)(s,s)\). This approach relies on the well-posedness of classical PDE solutions and It\^{o}'s formula, while it provides only global existence results for adapted solutions of Markovian BSVIEs. To establish the uniqueness of adapted solutions for a Markovian version of \eqref{General BSVIE}, the current study (\cite{Wang2021}) still imposes a restrictive condition that requires the generator to be linear in \( (Y,Z)(t,s) \). 
\end{itemize}
\end{enumerate}

The Lipschitz conditions assumed in \cite{hernandez2023me, Hernandez2021a} are specifically designed to facilitate the application of standard fixed-point arguments for proving the well-posedness of solutions of the nonlinear BSVIE system \eqref{General BSVIE}-\eqref{SysYt}. One may think of leveraging linearization techniques to extend the results in \cite{hernandez2023me, Hernandez2021a} from linear to nonlinear generator. However, this approach is not feasible without a new analytical framework, as it fails to address a core issue: the generator \( \nabla g \) of \eqref{SysYt} is non-Lipschitz. The analytical pain point is that fixed-point arguments require Lipschitz assumptions while treating \( Z(s,s) \) leads to non-Lipschitz \( (Y,Z)_t \)-BSVIE \eqref{SysYt}.

The discussions above review the literature while identify a key challenge in solving general BSVIEs \eqref{General BSVIE}: ensuring the well-definedness of the diagonal process \(\{Z(s,s)\}_{0\leq s\leq T}\) of \eqref{General BSVIE} in a non-Markovian setting requires enhancing regularity in the \(t\)-direction to achieve differentiability, which necessitates the investigation of \eqref{SysYt} for \((Y,Z)_t\). On the other hand, applying the chain rule to \eqref{General BSVIE} reveals that \eqref{SysYt} is non-Lipschitz. Solving a \(Z(s,s)\)-dependent BSVIE thus inherently involves addressing a non-Lipschitz BSVIE system of \eqref{General BSVIE}-\eqref{SysYt}. 

In summary, though BSVIEs \eqref{General BSVIE} with nonlinear dependence on $(Y,Z)(t,s)$ and $(Y,Z)(s,s)$ were attempted with certain degeneration, the well-posedness of the solutions to a general BSVIE \eqref{General BSVIE} remains an open problem. Addressing this challenge is the central contribution of this paper. Our success is built on top of the following insights:
\begin{enumerate}
    \item Inspired by the extensions of classic Lipschitz-based BSDE theories, such as stochastic Lipschitz, non-locally Lipschitz, and quadratic BSDEs (see \cite{zhang2017backward,pardoux2014stochastic}), if we consider generators without sufficient regularity (e.g., in the non-uniformly Lipschitz case), we have to appropriately restrict the search space for potential solutions. By focusing on solutions within a smaller while more regular space, we can establish their existence and uniqueness. In the case with the general BSVIE \eqref{General BSVIE}, equivalent to considering the non-Lipschitz BSVIE system \eqref{General BSVIE}-\eqref{SysYt}, our goal is to identify more regular and bounded (continuous) solutions, rather than merely integrable ones. In other words, we must strike a balance between the intrinsic parameters $(\xi,g)$ of \eqref{General BSVIE} and its adapted solution \((Y, Z)\).
    \item We leverage the Malliavin differentiability to ensure that the martingale integrand components \((Z, Z_t)(t, s)\) can be expressed in terms of the trace of the Malliavin derivative of \((Y, Y_t)(t, s)\). Specifically, \((Z, Z_t)(t, s) = D_s(Y, Y_t)(t, s)\). We will see that by focusing on bounded solutions for the BSDE system \eqref{SysYYt}-\eqref{MalliavinYYt}, all \( Y \) and its related derivative processes, including the $t$-directional derivative \( Y_t \) and Malliavin derivatives \( D_\theta Y \) and \( D_\theta Y_t \) will all be bounded. Consequently, one can find that \( Z \) and its $t$-derivative processes \( Z_t \) are likewise bounded. The boundedness of these solution processes and the martingale integrand processes greatly simplifies the complexity introduced by the non-linearity and non-Lipschitzness of \eqref{SysYt}. It is noteworthy that unlike many studies on the well-posedness of BSDEs with quadratic growth in \( Z \), where the comparison principle plays a crucial role for finding bounded solutions and thus restricts \( Y \) to being scalar, we have no such limitation while we allow \( Y \) and \( Z \) of \eqref{General BSVIE} to be of arbitrary dimensions. 
    \item A merit of leveraging a Malliavin calculus approach is that we can show \((Y, Z)(t, s)\) to be almost surely continuous in the \((t, s)\)-direction under suitable regularity assumptions on \((\xi, g)\). As a result, the intuitive definition of the diagonal processes \((Y,Z)(s, s):=(Y,Z)(s, \tau)|_{\tau=s}=(Y,D_\tau Y)(s, \tau)|_{\tau=s}=\lim_{\tau\to s^+}(Y,D_\tau Y)(s,\tau)\) becomes well-defined. This result differs from earlier approaches in \cite{Hamaguchi2021, hernandez2023me, Hernandez2021a}, where diagonal processes were defined using \(t\)-directional regularity. It also contrasts with the method of constructing adapted solutions in the Markovian case using PDEs, as seen in \cite{Wang2021, lei2023nonlocal, LeiSDG, lei2023well}. The continuity of \(s \mapsto (Y, Z)(t, s)\) is also crucial for deducing the rate of convergence for numerical schemes, as highlighted in \cite{hu2011malliavin}. Therefore, this paper contributes to the BSVIE field by not only proving the well-posedness of general BSVIEs but also offering a new perspective on dealing with the diagonal process \(\{(Y, Z)(s, s)\}_{0 \leq s \leq T}\) of \eqref{General BSVIE}.
    
\end{enumerate}

The main contribution of this paper is threefold: (1) This paper establishes, for the first time, the well-posedness of solutions to general BSVIEs, where the generator exhibits nonlinear dependence on both the solution process, the martingale integrand, and their diagonal processes. The methodology we used presents a novel perspective in handling the diagonal processes (via Malliavin calculus). (2) We connect BSVIE solutions defined with a diffusion process with classical solutions of corresponding semi-linear PDEs. Specifically, the unique adapted solution of the Markovian BSVIE provides a probabilistic interpretation of the unique classical solution of the associated PDE, rending a Feynman-Kac-type result, which has been found to be crucial in developing Monte-Carlo-based deep-learning solver for (high-dimensional) BSVIE and PDEs (see \cite{Beck2019}).
(3) We employ BSVIEs to establish dynamically optimal controls for mean-variance portfolio selection problems under various stochastic volatility models. Our well-posedness results guarantee that these equilibrium solutions are both well-defined and solvable. Additionally, we observe that the myopic and intertemporal hedging demands of the dynamically optimal investment policy are characterized by the two diagonal processes of the solution and the martingale integrand components in the BSVIE solutions, respectively. Our rigorous treatment of the diagonal processes enables us to adopt the dynamically optimal portfolios within a stochastic investment environment.

The remainder of this paper is organized as follows. Section \ref{sec:Well} first introduces the specific norms and Banach spaces suited for general BSVIEs and then examines the differentiation of BSVIE solutions with respect to both the parameter \( t \) and the Wiener space. On top of this, we apply Banach's fixed-point theorem to establish our main theoretical result: the existence, uniqueness, and regularity of BSVIE solutions. Some well-posedness results are also appropriately extended to more general settings. Section \ref{sec:MarkovianBSVIEs} applies our theoretical results to Markovian BSVIEs defined on a forward diffusion process. We provide an explicit representation of the Malliavin derivatives, followed by connecting between the adapted solutions of Markovian BSVIEs and the classical solutions of the corresponding semi-linear PDEs, extending the well-known Feynman-Kac formula to a nonlocal setting. Besides, our BSVIE well-posedness results support the study of TIC stochastic control problems in Section \ref{Sec: TICAPP}. Specifically, we use BSVIEs to formulate dynamically optimal mean-variance portfolio policies in incomplete markets. Finally, Section \ref{Sec: Conclusions} concludes.


\section{Well-posedness of Solutions of BSVIEs} \label{sec:Well}
This section is devoted to prove the existence, uniqueness, and regularity of solutions to general BSVIEs \eqref{General BSVIE} over an arbitrary time horizon. Due to the distinct structure of BSVIEs, i.e., its dependence on diagonal processes, we need to specify an appropriate space for a rigorous well-posedness analysis. We begin with defining specialized norms and Banach spaces for BSVIEs and then examine differentiability properties, including $t$-directional derivatives and Malliavin differentiation in the Wiener space. Then, we establish the well-posedness results and provide some extensions.


\subsection{Norms and Banach Spaces} \label{subsec:Norms}
Before introducing the Banach spaces for studying BSVIEs, we first introduce the following essential components, particularly the norms commonly used in previous studies on BSDEs:
\begin{itemize}
    \item \(C^l_b(\mathbb{R}^d; \mathbb{R}^k)\): the set of \(l\)-times continuously differentiable functions from \(\mathbb{R}^k\) to \(\mathbb{R}^q\), where both the functions and their partial derivatives up to order \(l\) are bounded;
    \item $L^p(0,T;\mathbb{R}^k)$: the set of Lebesgue measurable functions \(f: (0, T) \to \mathbb{R}^k\) such that $\int_0^T | f(s) |^p \, ds < \infty$ for \(1 \leq p < \infty\);
    \item \(L^\infty(0, T; \mathbb{R}^k)\): the set of Lebesgue measurable functions \(f: (0, T) \to \mathbb{R}^k\) that are essentially bounded;
    \item $L^p_{\mathcal{G}}(\Omega;\mathbb{R}^k)$: the set of all \(\mathbb{R}^k\)-valued $\mathcal{G}$-measurable random variable \(\xi\) such that \(\mathbb{E}[|\xi|^p] < \infty\);
    \item \(L^\infty_{\mathcal{G}}(\Omega; \mathbb{R}^k)\): the set of all \(\mathbb{R}^k\)-valued \(\mathcal{G}\)-measurable essentially bounded random variables;
    \item $L^p_{\mathbb{F}}(0,T;\mathbb{R}^k)$: the space of all \(\mathbb{F}\)-adapted \(\mathbb{R}^k\)-valued processes \(\{X(s)\}_{0\leq s\leq T}\) such that \(\mathbb{E}\big[ \int_0^T | X(s) |^p ds\big] < \infty\);
    \item \(L^\infty_{\mathbb{F}}(0, T; \mathbb{R}^d)\): the space of all \(\mathbb{F}\)-adapted \(\mathbb{R}^k\)-valued essentially bounded processes,
\end{itemize}
where \( | \cdot | \) denotes the Euclidean norm: for a vector \(u \in \mathbb{R}^k \), \(|u|=\big( \sum_{i=1}^k u_i^2 \big)^{1/2} \), and for a matrix \( u \in \mathbb{R}^{k \times n} \), \(|u|=\big( \sum_{j=1}^n \sum_{i=1}^k u_{ij}^2 \big)^{1/2} \). If needed, scalars or vectors can be treated as matrices, allowing us to use the notation \( | \cdot | \) with slight abuse. Thus, the spaces above can be naturally extended to matrix-valued (or Banach-space-valued) random variables or stochastic processes. Additionally, we introduce BMO (bounded mean oscillation) martingales and their associated norm, which are commonly used in the literature on quadratic BSDEs, as discussed in \cite{kobylanski2000backward, zhang2017backward}:
\begin{itemize}
    \item $\textsc{BMO}^2_\mathbb{F}(0,T;\mathbb{R}^{k})$: the space of square integrable martingales $\Phi$ with $\Phi_0=0$ and
    \begin{equation}  \label{BMO}
    \|\Phi\|^2_{\textsc{BMO}}=\sup\limits_\tau\left\|\mathbb{E}\left[\left.\langle\Phi\rangle_T-\langle\Phi\rangle_\tau\right|\mathcal{F}_\tau\right]\right\|_\infty<\infty,
    \end{equation}
    where the supremum is taken over all stopping times $\tau\in[0,T]$ and $\langle\Phi\rangle$ is the quadratic variation (process) of $\Phi$. 
    \item $S^\textsc{BMO}_\mathbb{F}(0,T;\mathbb{R}^{k\times n})$: the space of all \(\mathbb{F}\)-adapted \(\mathbb{R}^{k\times n}\)-valued processes $\{Z(s)\}_{0\leq s\leq T}$ such that $\int^\cdot_0Z(s)dB(s)\in\textsc{BMO}^2_\mathbb{F}(0,T;\mathbb{R}^k)$ with the norm $\|Z\|^2_{\mathcal{S}^{\textsc{BMO}}}=\|\int Z dB\|^2_{\textsc{BMO}}$.   
\end{itemize}

A BMO martingale is a martingale with bounded mean oscillation, meaning its quadratic variation process is uniformly controlled. The BMO norm quantifies this oscillation, measuring how much the martingale deviates from its mean. The space \( S^\textsc{BMO}_\mathbb{F}(0,T;\mathbb{R}^{k\times n}) \) is closely related to BMO martingales. We now discuss two key properties of this space to deepen our understanding.
\begin{lemma}
     If $\int^T_0|Z(s)|^2ds$ is integrable, then 
    \begin{equation} \label{SBMO}
        \|Z\|^2_{\mathcal{S}^{\textsc{BMO}}}=\sup\limits_{0\leq s\leq T}\Big\|\mathbb{E}\Big[\int\nolimits_s^T |Z(s) |^2\Big|\mathcal{F}_s\Big]\Big\|_\infty. 
    \end{equation}
    If $Z\in S^\textsc{BMO}_\mathbb{F}(0,T;\mathbb{R}^{k\times n})$ then one has automatically that $Z\in L^p_{\mathbb{F}}(0,T;\mathbb{R}^{k\times n})$ for all $p\geq 2$. Furthermore, the inclusion $L^\infty_\mathbb{F}\subsetneq S^{\textsc{BMO}}_\mathbb{F}\subsetneq \cap_{p>1}L^{p}_\mathbb{F}$ holds. 
    
\end{lemma}

The key difference between \eqref{BMO} and \eqref{SBMO} lies in the supremums taken over their respective domains. Specifically, the supremum over \( \{ 0 \leq s \leq T \} \) is smaller than that over \( \{ 0 \leq \tau \leq T \} \). A proof of this inverse inequality can be found in \cite[Lemma 1.1]{song2017multi}. Additionally, the BMO-associated space \( \mathcal{S}^\textsc{BMO}_\mathbb{F}(0,T;\mathbb{R}^{k \times n}) \) exhibits sufficient integrability while it is small but still slightly larger than the bounded space. We refer the readers to \cite{kazamaki2006continuous, frei2013quadratic} for more details.


\subsection*{Malliavin differentiation on the Wiener space} 
In addition to examining \( t \)-directional differentiability to ensure that the diagonal process \( \{ Z(s,s) \}_{0 \leq s \leq T} \) is well-defined, we shall also explore the differentiability properties for solutions of BSVIEs on the Wiener space. To this end, we provide some preliminaries on the Malliavin calculus and refer the readers to \cite{nualart2006malliavin, hu2011malliavin, Karoui1997} for further details.

Let $\mathcal{H}$ denote the set of random variables $\eta$ of the form $\eta=\varphi(B(h^1),B(h^2),\cdots,B(h^k))$, where $\varphi\in C^\infty_b(\mathbb{R}^k,\mathbb{R})$, $h^1,h^2,\cdots,h^k\in L^2([0,T];\mathbb{R}^n)$, and $B(h^i)=\int^T_0 h^i(s)dB(s)$. If $\eta\in\mathcal{H}$ is of the above form, we define its derivative as being the $n$-dimensional process 
\begin{equation*}
    D_\theta\eta=\sum\limits_{j=1}^k\partial_i\varphi(B(h^1),\cdots,B(h^k))h^j(\theta), \quad 0\leq \theta\leq T,  
\end{equation*}
where $\partial_i\varphi$ denotes the partial derivative of \(\varphi\) with respect to its \(i\)-th argument. It has been shown in\cite{nualart2006malliavin} that the operator $D$ has a closed extension to the space $\mathbb{D}^{1,2}$, the closure of $\mathcal{H}$ with respect to the norm $\|\cdot\|_{1,2}$ defined by $\|\eta\|^2_{1,2}:=\mathbb{E}\Big[|\eta|^2+\int\nolimits_0^T|D_\theta\eta|^2d\theta\Big]$ for $\eta\in\mathcal{H}$, $p>1$. 
It is clear that $D_\theta\eta=0$ for $\theta\in(s,T]$ if $\eta$ is $\mathcal{F}_s$-measurable. Furthermore, we use $L^2(0,T;(\mathbb{D}^{1,2})^k)$ to denote the set of $\mathbb{R}^k$-valued progressively measurable processes $u=\{u(s)\}_{0\leq s\leq T}$ such that: 
\begin{enumerate}
    \item For almost all $s\in[0,T]$, $u(s)\in(\mathbb{D}^{1,2})^k$. 
    \item $(s,\omega)\to Du(s,\omega)\in (L^2(0,T;\mathbb{R}))^{n\times k}$ admits a progressively measurable version.
    \item $\mathbb{E}\big[\int^T_0|u(s)|^2ds+\int^T_0\int^T_0|D_\theta u(s)|^2d\theta ds\big]<\infty$. 
\end{enumerate}
Observe that for each $(\theta,s,\omega)$, $D_\theta u(t,\omega)$ is an $n\times k$ matrix. Clearly, $D_\theta u(t,\omega)$ is defined uniquely up to sets of $d\theta\times ds\times dP$ of measure zero. 

\subsection{Banach Spaces for BSVIE Solutions}
We are now ready to introduce the norms and Banach spaces tailored for studying BSVIEs \eqref{General BSVIE}:
\begin{itemize}
    \item $\Lambda^\infty_{\mathcal{F}_T}([0,T]\times\Omega;\mathbb{R}^k)$: the set of bounded \( \mathbb{R}^k \)-valued \( \mathcal{F}_T \)-measurable random variables \( \xi(t) \) with regular \( t \)-directional and Malliavin differentiability satisfying
    \begin{equation*}
        \sup\limits_{0\leq\theta\leq T}\sup\limits_{0\leq t \leq T}\big\{\|\xi(t)\|_{\infty}+\|\xi_t(t)\|_{\infty}+\|D_\theta\xi(t)\|_{{\infty}}+\|D_\theta\xi_t(t)\|_{\infty}\big\}<\infty.
    \end{equation*}
    \item $\Lambda^\infty_{\mathbb{F},c}(\Delta[t_0,T]\times\Omega;\mathbb{R}^k)$: the set of \( \mathbb{F} \)-adapted \( \mathbb{R}^k \)-valued processes \( \{ Y(t, s) \}_{t_0 \leq t \leq s \leq T} \), which are \( t \)-directional and Malliavin differentiable, with mixed partial derivatives existing, bounded in \( (t, s, \theta) \), continuous in \( s \), and satisfying
    \begin{equation*}
        \sup\limits_{0\leq\theta\leq T}\sup\limits_{t_0\leq t\leq s\leq T}\big\{\|Y(t,s)\|_{\infty}+\|Y_t(t,s)\|_{\infty}+\|D_\theta Y(t,s)\|_{{\infty}}+\|D_\theta Y_t(t,s)\|_{\infty}\big\}<\infty.
    \end{equation*}
    \item \(\Theta^{\textsc{BMO}}_{\mathbb{F}}(\Delta[t_0,T]\times\Omega;\mathbb{R}^{k\times n})\): the set of \( \mathbb{F} \)-adapted \( \mathbb{R}^{k \times n} \)-valued processes \( \{ Z(t, s) \}_{t_0 \leq t \leq s \leq T} \), where these processes, along with their \( t \)-directional and Malliavin derivatives, are required to be \( L^2 \)-integrable with respect to \( s \), be uniformly bounded in the other arguments, and satisfy
    \begin{equation*}
        \begin{split}
            & \sup\limits_{0\leq\theta\leq T}\sup\limits_{t_0\leq t\leq s\leq T}\Big\{\Big\|\mathbb{E}\Big[\int\nolimits_s^T|Z(t,\tau)|^2d\tau\Big|\mathcal{F}_s\Big]\Big\|_\infty+\Big\|\mathbb{E}\Big[\int\nolimits_s^T|Z_t(t,\tau)|^2d\tau\Big|\mathcal{F}_s\Big]\Big\|_\infty \\
            & \qquad\qquad 
            +\Big\|\mathbb{E}\Big[\int\nolimits_s^T|D_\theta Z(t,\tau)|^2d\tau\Big|\mathcal{F}_s\Big]\Big\|_\infty+\Big\|\mathbb{E}\Big[\int\nolimits_s^T|D_\theta Z_t(t,\tau)|^2d\tau\Big|\mathcal{F}_s\Big]\Big\|_\infty \Big\}<\infty.
        \end{split}
    \end{equation*}
    Similarly, one can define a set of vector-valued processes, \(\Theta^{\textsc{BMO}}_{\mathbb{F}}(\Delta[t_0,T]\times\Omega;\mathbb{R}^{k})\). 
    \item \(\Theta^{\textsc{BMO}}_{\mathbb{F},b}(\Delta[t_0,T]\times\Omega;\mathbb{R}^{k\times n})\): a subspace of \( \Theta^{\textsc{BMO}}_{\mathbb{F}}(\Delta[t_0,T] \times \Omega; \mathbb{R}^{k \times n}) \), where the mapping \( s \mapsto (Z, Z_t)(t,s) \) is required to be not only integrable but also uniformly bounded across \( \omega \in \Omega \). This subspace is equipped with the norm
    \begin{equation*}
        \begin{split}
            & \sup\limits_{0\leq\theta\leq T}\sup\limits_{t_0\leq t\leq s\leq T}\Big\{\|Z(t,s)\|_{\infty}+\|Z_t(t,s)\|_{\infty} \\
            & \qquad\qquad 
            +\Big\|\mathbb{E}\Big[\int\nolimits_s^T|D_\theta Z(t,\tau)|^2d\tau\Big|\mathcal{F}_s\Big]\Big\|_\infty+\Big\|\mathbb{E}\Big[\int\nolimits_s^T|D_\theta Z_t(t,\tau)|^2d\tau\Big|\mathcal{F}_s\Big]\Big\|_\infty \Big\}<\infty.
        \end{split}
    \end{equation*}
    \end{itemize}
    
    One of the most distinct features of BSDEs is that their solutions consist of two adapted processes. Therefore, let us consider some product spaces formed from the Banach spaces:
    \begin{itemize}
    \item $\mathcal{N}[t_0,T]$: the product space $(\Lambda_c\otimes\Theta)[t_0,T]$ defined by  
    \begin{equation*}
    (\Lambda_c\otimes\Theta)[t_0,T]=\Lambda^\infty_{\mathbb{F},c}(\Delta[t_0,T]\times\Omega;\mathbb{R}^k)\times\Theta^{\textsc{BMO}}_{\mathbb{F}}(\Delta[t_0,T]\times\Omega;\mathbb{R}^{k\times n}), 
    \end{equation*}
    equipped with the norm
    \begin{equation*}
    \|(Y,Z)\|^2_{\mathcal{N}[t_0,T]}=\|Y\|^2_{\Lambda^\infty_{\mathbb{F},c}(\Delta[t_0,T])}+\|Z\|^2_{\Theta^{\textsc{BMO}}_{\mathbb{F}}(\Delta[t_0,T])}. 
    \end{equation*}
    
    \item $\mathcal{M}[t_0,T]$: a sub-product space of $\mathcal{N}[t_0,T]$ defined by  
    \begin{equation*}
        (\Lambda_c\otimes\Theta_b)[t_0,T]=\Lambda^\infty_{\mathbb{F},c}(\Delta[t_0,T]\times\Omega;\mathbb{R}^k)\times\Theta^{\textsc{BMO}}_{\mathbb{F},b}(\Delta[t_0,T]\times\Omega;\mathbb{R}^{k\times n}),
    \end{equation*}
    equipped with the norm
    \begin{equation*}
        \|(Y,Z)\|^2_{\mathcal{M}[t_0,T]}=\|Y\|^2_{\Lambda^\infty_{\mathbb{F},c}(\Delta[t_0,T])}+\|Z\|^2_{\Theta^{\textsc{BMO}}_{\mathbb{F},b}(\Delta[t_0,T])}.
    \end{equation*}
    
    \item 
    $\mathcal{M}_c[t_0,T]:=\big\{(Y,Z)\in\mathcal{M}[t_0,T]:s\mapsto(Y,Z)(t,s)\text{~is~continuous}\big\}$. 
\end{itemize}

More generally, suppose \( \{\mathcal{B}_i, i \in I\} \) is a collection of Banach spaces with norms \( \|\cdot\|_i \). The Cartesian product \( \otimes_i \{\mathcal{B}_i\} \) is a Banach space with the norm \( \|v\| \equiv \sum_i \|v_i\|_i \). Clearly, these tailored Banach spaces $\mathcal{N}$ and $\mathcal{M}$ for BSVIEs \eqref{General BSVIE} build on the previously defined normed spaces of bounded and BMO processes. They account not only for the pointwise and integrability regularities of the stochastic processes but also for their \( t \)-directional and Malliavin differentiability. The regularity conditions in all directions ensure that the BSVIEs \eqref{General BSVIE} in these spaces are well-behaved and well-posed.



\subsection{Differentiation of BSVIE Solutions} \label{subsec:Diff}
To prove the well-posedness of BSVIE solutions, we conduct some prior analyses by presuming the existence of a solution to \eqref{GeneralCase} in the space \(\mathcal{M}[0,T]\). We then examine a stochastic system satisfied by the adapted solution pair \((Y,Z)(t,s)\) along with its \(t\)-directional derivative \((Y,Z)_t(t,s)\), Malliavin derivative \(D_\theta(Y,Z)(t,s)\), and their mixed derivatives \(D_\theta(Y,Z)_t(t,s)\). 

Let us consider a general BSVIE of the form:
\begin{equation} \label{GeneralCase}
    \begin{cases}
        \begin{aligned}
            dY(t,s) & ~=~ g(t,s,Y(t,s),Z(t,s),Y(s,s),Z(s,s))ds+Z(t,s)dB(s), \\
       Y(t,T) & ~=~ \xi(t), \quad 0\leq t \leq s\leq T, 
        \end{aligned}  
    \end{cases}
\end{equation} 
where the random generator $g:\Delta[0,T]\times\Omega\times\mathbb{R}^{2k}\times\mathbb{R}^{2(k\times n)} \to \mathbb{R}^k$ and the $\mathcal{F}_T$-measurable terminal data \(\xi:[0,T]\times\Omega\to\mathbb{R}^k\) satisfy the following assumptions: 
\begin{enumerate}[start=1,label={\upshape(\bfseries A\arabic*).}]
    \item \(\xi(t)\in\Lambda^\infty_{\mathcal{F}_T}([0,T]\times\Omega;\mathbb{R}^k)\). \label{A1}  
    \item \( g(t, s, y, z, \overline{y}, \overline{z}) \) is Lipschitz continuous in \( (y, \overline{y}, z, \overline{z}) \in \mathbb{R}^{2k} \times \mathbb{R}^{2(k \times n)} \), uniformly in \( (t, s, \theta, \omega) \in \Delta[0,T]\times[0,T]\times \Omega \). Moreover, its first-order derivatives marked with \( \surd_L \) in Table \ref{tab:driver1} are uniformly Lipschitz in \( (y, \overline{y}, z, \overline{z}) \), while those marked with \( \surd_{BL} \) are uniformly bounded across all arguments and uniformly Lipschitz in \( (y, \overline{y}, z, \overline{z}) \). The requirements of same interpretations apply to second-order derivatives marked with \( \surd_{L} \) and \( \surd_{BL} \) in Table \ref{tab:driver2}. Let \( L > 0 \) denote the maximum of all these Lipschitz constants and bounds. \label{A2}
    \item $g(t,s,0,0,0,0)\in\Theta^{\textsc{BMO}}_{\mathbb{F}}(\Delta[0,T]\times\Omega;\mathbb{R}^{k})$. \label{A3} 
\end{enumerate}

\begin{table}[!ht] 
	\centering
    \begin{tabular}{c| c c c c c c c c c}
       \hline
       $\alpha$ & $t$ & $D_\theta$ & $s$ & $y$ & $z$ & $\overline{y}$ & $\overline{z}$  \\ 
      \hline 
      $g_\alpha$ & $\surd_L$ & $\surd_L$ & & $\surd_{BL}$ & $\surd_{BL}$ & $\surd_{BL}$ & $\surd_{BL}$ \\ 
      \hline 
    \end{tabular}
	\caption{First-order derivatives of $g$ required to be bounded and Lipschitz continuous}
	\label{tab:driver1}
\end{table} 

\begin{table}[!ht] 
  \centering
  \begin{tabular}{c| c c c c c c c c c}
  	\hline
  	\diagbox{$\alpha$}{$g_{\alpha\beta}$}{$\beta$} & $t$ & $D_\theta$ & $s$ & $y$ & $z$ & $\overline{y}$ & $\overline{z}$ \\ 
  	\hline 
  	$t$ &  & $\surd_{L}$ &  & $\surd_{BL}$ & 
        $\surd_{BL}$ & $\surd_{BL}$ & $\surd_{BL}$ \\
        $D_\theta$ & $\surd_{L}$ &  &  & $\surd_{BL}$ & $\surd_{BL}$ &  &  \\
  	$s$ &  &  &  &  & \\
  	$y$ & $\surd_{BL}$ & $\surd_{BL}$ &  & $\surd_{BL}$ & $\surd_{BL}$ & $\surd_{BL}$ & $\surd_{BL}$ \\
  	$z$ & $\surd_{BL}$ & $\surd_{BL}$ &  & $\surd_{BL}$ & $\surd_{BL}$ & $\surd_{BL}$ & $\surd_{BL}$ \\
  	$\overline{y}$ & $\surd_{BL}$ &  &  & $\surd_{BL}$ & $\surd_{BL}$  &  & \\ 
   $\overline{z}$ & $\surd_{BL}$ &  &  & $\surd_{BL}$ & $\surd_{BL}$ & &  \\
  	\hline 
  \end{tabular}
  \caption{Second-order derivatives of $g$ required to be bounded and Lipschitz continuous}
  \label{tab:driver2}
\end{table} 

The assumptions (A1)-(A3) were amended for \eqref{GeneralCase} from the standard ones in BSDE studies; see \cite{Pardoux1990, Yong1999, zhang2017backward, pardoux2014stochastic}. Specifically, (A1) imposes regularity on the terminal datum \( \xi \), which reflects the regularity of the solution \( Y \), since \(\xi(t) = Y(t, T)\). (A2) requires Lipschitz continuity for the generator \( g \) in its arguments with continuous differentiability being a stricter form. Due to the dependence on diagonal processes in BSVIEs \eqref{GeneralCase}, the generator \( g \) needs to be beyond standard Lipschitz continuity, as noted in BSDE literature; see \cite{Pardoux1990}. In particular, we need its \( t \)-directional and Malliavin differentiability, along with existence of mixed derivatives. Hence, \( g \) needs to possess continuous first- and second-order derivatives as detailed in Tables \ref{tab:driver1} and \ref{tab:driver2}.  Finally, since we are working in a non-Markovian setting, the generator \( g \) may explicitly depend on the trajectory \(\omega \in \Omega\). Therefore, it is natural to impose an integrability condition (A3) on \( g \) at \( (t,s,0,0,0,0) \), which, together with (A2), effectively controls the growth of \( g(t,s,y,z,\overline{y},\overline{z}) \). Similar requirements, outlined in Tables \ref{tab:driver1}-\ref{tab:driver2}, can also be found in PDE studies, as noted in \cite{lei2023nonlocal, LeiSDG, lei2023well}. Given the close connection between BSVIEs and PDEs, highlighted by the Feynman--Kac formula, we expect a parallel PDE theory to our study and it will be explored in Section \ref{sec:MarkovianBSVIEs}.

These conditions are easily satisfied by functions with continuous, bounded third-order derivatives. It is noteworthy that in contrast to directly applying Lipschitz assumption on \(\nabla g = g_t + g_Y \cdot Y_t(t,s) + g_Z \cdot Z_t(t,s)\) in \eqref{SysYt} with respect to \( (Y,Z)(t,s) \), \( (Y,Z)(s,s) \), and \( (Y,Z)_t(t,s) \) as in \cite{hernandez2023me, Hernandez2021a}, the condition (A2) is more reasonable and verifiable because it imposes Lipschitz conditions on \( g_t \), \( g_Y \), and \( g_Z \) without requiring any prior information about the undetermined solution \( (Y,Z)_t(t,s) \). To clarify our approach, we have chosen to strengthen our assumptions to simplify the proofs. Subsequently, we will show that (A2) can be significantly relaxed to accommodate non-uniform Lipschitz continuity (A2') and (A2''), or even local Lipschitz conditions; see Remark \ref{RemkLocallyL}.

Let us assume that there exists an adapted solution $(Y,Z)\in\mathcal{M}[0,T]$ satisfying \eqref{GeneralCase}. Under assumptions (A1)-(A3) about $g$ and $\xi$, it is clear that the generator $g$ evaluated at $(t,s,Y(t,s),Z(t,s),Y(s,s),Z(s,s))$ is $t$-directional differentiable. Consequently, Proposition 2.4 of \cite{Karoui1997} shows that the function $t\mapsto(Y,Z)(t,s)$ is differentiable with derivatives given
by $(Y,Z)_t(t,s)$. By noting the integral representation of diagonal process pair,
\begin{equation} \label{IntegralRepre}
    (Y,Z)(s,s)=(Y,Z)(t,s)+\int^s_t(Y,Z)_t(\alpha,s)d\alpha, 
\end{equation}
one has a coupled BSDE system for $(Y,Z)(t,s)$ and $(Y,Z)_t(t,s)$, 
\begin{equation} \label{SysYYt}
    \begin{cases}
        \begin{aligned}
            dY(t,s) & ~=~ \overline{g}\Big(t,s,Y(t,s),Z(t,s),\int^s_tY_t(\alpha,s)d\alpha,\int^s_tZ_t(\alpha,s)d\alpha\Big)ds +Z(t,s)dB(s), \\

        dY_t(t,s) & ~=~ \Big[\overline{g}_t+\overline{g}_Y\cdot Y_t(t,s) +\overline{g}_Z\cdot Z_t(t,s)\Big]ds+Z_t(t,s)dB(s), \\

       Y(t,T) & ~=~ \xi(t), \quad Y_t(t,T) ~=~ \xi_t(t), \quad 0\leq t \leq s\leq T,
        \end{aligned} 
    \end{cases}
\end{equation}
where \( g_t \), \( g_Y \), and \( g_Z \) are the partial derivatives of \( g(t,s,Y(t,s),Y(s,s),Z(t,s),Z(s,s)) \) from \eqref{GeneralCase} with respect to \( t \), \( Y(t,s) \), and \( Z(t,s) \), respectively. Furthermore, the overlines on \( g_t \), \( g_Y \), and \( g_Z \) indicate that after taking the partial derivatives, we replace the arguments \( (Y,Z)(s,s) \) with \( (Y,Z)(t,s) \) and \( \int_t^s (Y,Z)_t(\alpha,s) \, d\alpha \) with the integral expression \eqref{IntegralRepre}, respectivelysi, i.e., similar to the way \( \overline{g} \) and \( g \) are used in the first equation of \eqref{SysYYt}. For illustration, 
\begin{equation*}
    \begin{split}
        & \overline{g}\Big(t,s,Y(t,s),Z(t,s),\int^s_tY_t(\alpha,s)d\alpha,\int^s_tZ_t(\alpha,s)d\alpha\Big) \\
        = & g\Big(t,s,Y(t,s),Z(t,s),Y(t,s)+\int^s_tY_t(\alpha,s)d\alpha,Z(t,s)+\int^s_tZ_t(\alpha,s)d\alpha\Big),
    \end{split}
\end{equation*}
which also demonstrates how \( (Y,Z)(t,s) \) and \( (Y,Z)_t(t,s) \) of \eqref{SysYYt} are coupled together. The Lipschitz continuity and bounded derivatives of \( g \) specified in Table \ref{tab:driver1}-\ref{tab:driver2} inherit the same mathematical properties for \( \overline{g} \), and vice versa. If the coefficient processes \( \overline{g}_Y \) and \( \overline{g}_Z \) of \eqref{SysYYt} are given, then \eqref{SysYYt} becomes a nonlocal BSDE, the nonlocality of which arises from integration over additional parameters, as discussed in \cite{lei2024MV}. Under suitable Lipschitz assumptions, \cite{lei2024MV} has established the well-posedness for such nonlocal BSDEs. However, when \( g_Y \) and \( g_Z \) depend on the undetermined solutions, we have to simultaneously solve for \( (Y, Z, Y_t, Z_t)(t,s) \) from this non-Lipschitz BSVIE system and this feature distinguishes our paper from earlier works in \cite{lei2024MV} and \cite{Hamaguchi2021,hernandez2023me, Hernandez2021a}.


Assuming that \( g \) and \( \xi \) are Malliavin differentiable on the Wiener space and that \( (Y, Z) \in \mathcal{M}[0, T] \), Proposition 5.3 of \cite{Karoui1997} establishes that the solutions \( (Y, Z)(t, s) \) and \( (Y, Z)_t(t, s) \) of \eqref{SysYYt} are also Malliavin differentiable with derivatives given by \( D_\theta(Y, Z)(t, s) \) and \( D_\theta(Y, Z)_t(t, s) \), i.e., \( (Y, Z), (Y, Z)_t \in L^2(0, T; (\mathbb{D}^{1,2})^k \times (\mathbb{D}^{1,2})^{n \times k}) \). A simple application of the chain rule and product rule to \eqref{GeneralCase} (or \eqref{SysYYt}) as in \cite{Pardoux1992, Karoui1997, hu2011malliavin, yong2008well, Wang2020} leads to a straightforward result: a coupled BSDE system \eqref{MalliavinYYt} of \(D_\theta(Y,Z)(t,s)\) and \(D_\theta(Y,Z)_t(t,s)\):
\begin{equation} \label{MalliavinYYt}
    \begin{cases}



        \begin{aligned}
            dD_\theta Y(t,s) & ~=~ \Big[D_\theta \overline{g}+ \widetilde{g}_Y\cdot D_\theta Y(t,s) + \widetilde{g}_Z\cdot D_\theta Z(t,s) +\overline{g}_{\overline{Y}}\cdot \int^s_tD_\theta Y_t(\alpha,s)d\alpha \\ 

            & \qquad + \overline{g}_{\overline{Z}}\cdot\int^s_tD_\theta Z_t(\alpha,s)d\alpha\Big]ds+D_\theta Z(t,s)dB(s), \\

            dD_\theta Y_t(t,s) & ~=~ \bigg\{D_\theta \overline{g}_t + \widetilde{g}_{tY}\cdot D_\theta Y(t,s) + \widetilde{g}_{tZ}\cdot D_\theta Z(t,s)  + \overline{g}_{t\overline{Y}}\cdot \int^s_tD_\theta Y_t(\alpha,s)d\alpha  \\

            & + \overline{g}_{t\overline{Z}}\cdot \int^s_tD_\theta Z_t(\alpha,s)d\alpha +\Big[D_\theta \overline{g}_Y + \widetilde{g}_{YY}\cdot D_\theta Y(t,s) + \widetilde{g}_{YZ}\cdot D_\theta Z(t,s) \\
        
            & + \overline{g}_{Y\overline{Y}}\cdot \int^s_tD_\theta Y_t(\alpha,s)d\alpha  + \overline{g}_{Y\overline{Z}}\cdot \int^s_tD_\theta Z_t(\alpha,s)d\alpha\Big]\cdot Y_t(t,s) +\overline{g}_Y \cdot D_\theta Y_t(t,s) \\

            & +\Big[D_\theta \overline{g}_Z + \widetilde{g}_{ZY}\cdot D_\theta Y(t,s) + \widetilde{g}_{ZZ}\cdot D_\theta Z(t,s)  + \overline{g}_{Z\overline{Y}}\cdot \int^s_tD_\theta Y_t(\alpha,s)d\alpha \\
        
            & + \overline{g}_{Z\overline{Z}}\cdot \int^s_tD_\theta Z_t(\alpha,s)d\alpha\Big]\cdot Z_t(t,s) +\overline{g}_Z \cdot D_\theta Z_t(t,s)\bigg\}ds +D_\theta Z_t(t,s)dB(s), \\

            D_\theta Y(t,T) & ~=~ D_\theta\xi(t), \quad D_\theta Y_t(t,T) ~=~ D_\theta\xi_t(t), \quad 0\leq t \leq s\leq T, \quad 0\leq \theta\leq s\leq T, 
        \end{aligned}
    \end{cases}
\end{equation}
and $D_\theta Y(t,s) = 0$ and $D_\theta Z(t,s) = 0$ for $0\leq s< \theta \leq T$, where \( \widetilde{g}_{Y} = \overline{g}_{Y} + \overline{g}_{\overline{Y}} \), \( \widetilde{g}_{Z} = \overline{g}_{Z} + \overline{g}_{\overline{Z}} \), and \( D_\theta \overline{g} \), \( D_\theta \overline{g}_t \), \( D_\theta \overline{g}_Y \), and \( D_\theta \overline{g}_Z \) denote the Malliavin derivatives of $g$, $g_t$, $g_Y$, and $g_Z$ 
with \( (Y,Z)(s,s) \) replaced by \( (Y,Z)(t,s) + \int_t^s (Y,Z)_t(\alpha,s) \, d\alpha \) after differentiation. Moreover, \( \overline{g}_{tY} \), \( \overline{g}_{t\overline{Y}} \), \( \overline{g}_{tZ} \), and \( \overline{g}_{t\overline{Z}} \) are partial derivatives of \( g_t(t,s,Y(t,s),Y(s,s),Z(t,s),Z(s,s)) \) with respect to \( Y(t,s) \), \( Y(s,s) \), \( Z(t,s) \), and \( Z(s,s) \), respectively, where the diagonal processes are substituted using the integral expression \eqref{IntegralRepre}. Similarly, we define the partial derivatives of $g_Y$ and $g_Z$ by \( \overline{g}_{YY} \), \( \overline{g}_{Y\overline{Y}} \), \( \overline{g}_{YZ} \), \( \overline{g}_{Y\overline{Z}} \), \( \overline{g}_{ZY} \), \( \overline{g}_{Z\overline{Y}} \), \( \overline{g}_{ZZ} \), and \( \overline{g}_{Z\overline{Z}} \), which give rise to defining \( \widetilde{g}_{tY} = \overline{g}_{tY} + \overline{g}_{t\overline{Y}} \), \( \widetilde{g}_{tZ} = \overline{g}_{tZ} + \overline{g}_{t\overline{Z}} \), \( \widetilde{g}_{YY} = \overline{g}_{YY} + \overline{g}_{Y\overline{Y}} \), \( \widetilde{g}_{YZ} = \overline{g}_{YZ} + \overline{g}_{Y\overline{Z}} \), \( \widetilde{g}_{ZY}=\overline{g}_{ZY}+\overline{g}_{Z\overline{Y}} \), and \( \widetilde{g}_{ZZ}=\overline{g}_{ZZ}+\overline{g}_{Z\overline{Z}} \).

If \((Y,Z)(t,s)\) and \((Y,Z)_t(t,s)\) are presumed to be known (from solving \eqref{SysYYt}), then all coefficients in \eqref{MalliavinYYt} are fully determined and in this case, \eqref{MalliavinYYt} degenerates to a coupled linear BSDE system for \( D_\theta(Y,Z)(t,s) \) and \( D_\theta(Y,Z)_t(t,s) \). Note that the presence of \( (Y_t, Z_t) \) in the coefficients of \eqref{MalliavinYYt} suggests that if \eqref{SysYYt} were studied solely in a normed space of \( L^p \)-integrable processes, the difficulty would be raised to another level due to the need to analyze stochastic Lipschitz BSVIEs \eqref{MalliavinYYt}, as discussed in \cite{lei2024MV}. This motivates our focus on proving the well-posedness of \eqref{GeneralCase} in \( \mathcal{M}[0,T] \), where the boundedness of the \( (Y,Z)_t \) components eliminates the need to handle stochastic Lipschitz coefficients, thereby simplifying our analysis. From our earlier analysis, we observed that establishing the well-posedness of solutions to the general BSVIE \eqref{GeneralCase} in the highly regular Banach space \( \mathcal{M}[0,T] \) requires not only considering the original BSVIE \eqref{GeneralCase} but also examining the dynamics of \eqref{SysYYt} and \eqref{MalliavinYYt}, which characterize the various derivatives of the solution. As a result, the study of BSVIEs \eqref{GeneralCase} necessitates the investigation of the coupled system \eqref{SysYYt}-\eqref{MalliavinYYt}, where the generator is not standard Lipschitz. This pain point particularly makes the well-posedness analysis of BSVIEs \eqref{GeneralCase} challenging.


\subsection{Existence and uniqueness of BSVIE solutions} \label{subsec:Wellposedness}
In this subsection, we present and prove the main theoretical results of this paper: the well-posedness of solutions to BSVIEs \eqref{GeneralCase} over an arbitrarily large time horizon. Moreover, we further improve the results in terms of the regularity of the solutions with additional continuity assumptions on the Malliavin derivatives of $(\xi,g)$. While this subsection focuses on a Malliavin calculus approach in handling the diagonal processes within the well-posedness analysis, the next subsection presents appropriate extensions to a more general framework.

\begin{theorem} \label{GeneralWellp}
    Under assumptions $\mathrm{(A1)}-\mathrm{(A3)}$, the BSVIE \eqref{GeneralCase} admits a unique adapted solution $(Y,Z)(t,s)\in\mathcal{M}[0,T]$. 
\end{theorem}

\begin{proof}
Our proof is outlined and divided into the following three main steps:
\begin{enumerate}[start=1,label={\upshape \bfseries Step \arabic*.}]
    \item \textbf{Local existence.} We consider a closed ball \( B_{\delta,T}(R) \) centered at $(0,0)$ with radius \( R \) in \( \mathcal{M}[T-\delta, T] \). By choosing a small enough \( \delta \) and sufficiently large \( R \), we then construct a self-mapping contraction within this ball so as to ensure ensures the existence of a unique fixed point within this ball by Banach's fixed-point theorem.
    \item \textbf{Local uniqueness.} By means of contradiction arguments, we show that the uniqueness of the solution extends to the entire space \( \mathcal{M}[T-\delta, T] \) rather than just \( B_{\delta,T}(R) \).
    \item \textbf{Global well-posedness.} Finally, by updating the terminal time and datum, the local solution can be extended to a larger time interval. After obtaining an a-priori estimate for the solution, we can argue that this procedure can be repeated indefinitely, up to constructing a unique global solution to the BSVIE over an arbitrarily large time horizon.
\end{enumerate}


\noindent\textbf{Step 1.} First of all, we identify a solution of BSVIE \eqref{GeneralCase} in $\Delta[T-\delta,T]$ as a fixed point of the operator $\Gamma$ defined in a non-empty closed subspace $B_{\delta,T}(R)$ of $(Y,Z)(t,s)\in\mathcal{M}[T-\delta, T]$ 
    \begin{equation*}
        B_{\delta,T}(R):=\Big\{(Y,Z)(t,s)\in\mathcal{M}[T-\delta, T]:~\|(Y,Z)\|^2_{\mathcal{M}[T-\delta,T]}\leq R\Big\},  
    \end{equation*}
by $\Gamma((\overline{Y},\overline{Z}))=(Y,Z)$, where $\Gamma((Y,Z))$ is the solution to 
\begin{equation*}
    \begin{cases}
        \begin{aligned}
            dY(t,s) & ~=~ g\big(t,s,\overline{Y}(t,s),\overline{Z}(t,s),\overline{Y}(s,s),\overline{Z}(s,s)\big)ds+Z(t,s)dB(s), \\
        
       Y(t,T) & ~=~ \xi(t), \quad T-\delta\leq t \leq s\leq T. 
        \end{aligned} 
    \end{cases}
\end{equation*} 
It is important to note that the diagonal process pair \(\{(\overline{Y},\overline{Z})(s,s)\}_{T-\delta \leq s \leq T}\) is well-defined for every \((\overline{Y},\overline{Z})(t,s) \in \mathcal{M}[T-\delta, T]\) given by \( (\overline{Y},\overline{Z})(s,s) = (\overline{Y},\overline{Z})(t,s) + \int_t^s (\overline{Y},\overline{Z})_t(\alpha,s) \, d\alpha \). Moreover, under the regularity assumptions for \( (\overline{Y}, \overline{Z})(t,s) \) within this space as well as for \( g \) and \( \xi \) as outlined in (A1)-(A3) and as discussed in the previous subsection and in \cite{Karoui1997}, the solution $(Y,Z)(t,s)$, being the image of the mapping \( \Gamma \), inherits high regularity properties, including the \( t \)-directional differentiability and Malliavin differentiability. Consequently, for $T-\delta\leq t\leq s\leq T$ and $0\leq \theta\leq s\leq T$, it follows that
\begin{equation} \label{Mapping}
    \begin{cases}
        \begin{aligned}
            dY(t,s) & ~=~ \overline{g}\Big(t,s,\overline{Y}(t,s),\overline{Z}(t,s),\int^s_t\overline{Y}_t(\alpha,s)d\alpha,\int^s_t\overline{Z}_t(\alpha,s)d\alpha\Big)ds+Z(t,s)dB(s), \\

            dY_t(t,s) & ~=~ \Big[\overline{g}_t+\overline{g}_Y\cdot\overline{Y}_t(t,s) +\overline{g}_Z\cdot\overline{Z}_t(t,s)\Big]ds+Z_t(t,s)dB(s), \\

            dD_\theta Y(t,s) & ~=~ \Big[D_\theta \overline{g}+ \widetilde{g}_Y\cdot D_\theta \overline{Y}(t,s) + \widetilde{g}_Z\cdot D_\theta \overline{Z}(t,s) +\overline{g}_{\overline{Y}}\cdot \int^s_tD_\theta\overline{Y}_t(\alpha,s)d\alpha \\

            & \qquad + \overline{g}_{\overline{Z}}\cdot\int^s_tD_\theta\overline{Z}_t(\alpha,s)d\alpha\Big]ds+D_\theta Z(t,s)dB(s), \\

            dD_\theta Y_t(t,s) & ~=~\bigg\{D_\theta \overline{g}_t + \widetilde{g}_{tY}\cdot D_\theta \overline{Y}(t,s) + \widetilde{g}_{tZ}\cdot D_\theta \overline{Z}(t,s)  + \overline{g}_{t\overline{Y}}\cdot \int^s_tD_\theta\overline{Y}_t(\alpha,s)d\alpha  \\

            & \qquad + \overline{g}_{t\overline{Z}}\cdot \int^s_tD_\theta\overline{Z}_t(\alpha,s)d\alpha +\Big[D_\theta \overline{g}_Y + \widetilde{g}_{YY}\cdot D_\theta \overline{Y}(t,s) + \widetilde{g}_{YZ}\cdot D_\theta \overline{Z}(t,s) \\

            & \qquad + \overline{g}_{Y\overline{Y}}\cdot \int^s_tD_\theta\overline{Y}_t(\alpha,s)d\alpha  + \overline{g}_{Y\overline{Z}}\cdot \int^s_tD_\theta\overline{Z}_t(\alpha,s)d\alpha\Big]\cdot\overline{Y}_t(t,s) +\overline{g}_Y \cdot D_\theta\overline{Y}_t(t,s) \\
            & \qquad +\Big[D_\theta \overline{g}_Z + \widetilde{g}_{ZY}\cdot D_\theta \overline{Y}(t,s) + \widetilde{g}_{ZZ}\cdot D_\theta \overline{Z}(t,s)  + \overline{g}_{Z\overline{Y}}\cdot \int^s_tD_\theta\overline{Y}_t(\alpha,s)d\alpha  \\

            & \qquad + \overline{g}_{Z\overline{Z}}\cdot \int^s_tD_\theta\overline{Z}_t(\alpha,s)d\alpha\Big]\cdot \overline{Z}_t(t,s) +\overline{g}_Z \cdot D_\theta\overline{Z}_t(t,s)\bigg\}ds +D_\theta Z_t(t,s)dB(s), \\

            Y(t,T) & ~=~ \xi(t), \quad Y_t(t,T) ~=~ \xi_t(t),\quad D_\theta Y(t,T) ~=~ D_\theta\xi(t), \quad D_\theta Y_t(t,T) ~=~ D_\theta\xi_t(t),  
        \end{aligned}
    \end{cases}
\end{equation}
and $D_\theta Y(t,s)=D_\theta Z(t,s)=D_\theta Y_t(t,s)=D_\theta Z_t(t,s)=0$ for $s<\theta\leq T$.

Next, we will show that for every $(\overline{Y}^1,\overline{Z}^1)$ and $(\overline{Y}^2,\overline{Z}^2)\in B_{\delta,T}(R)$, by noting $\Delta(\overline{Y},\overline{Z})=(\overline{Y}^1-\overline{Y}^2,\overline{Z}^1-\overline{Z}^2)$, the following inequality holds 
\begin{equation} \label{Contraction}
    \|\Delta(Y,Z)\|^2_{\mathcal{M}[T-\delta,T]}\leq C(R)\delta\|\Delta(\overline{Y},\overline{Z})\|^2_{\mathcal{M}[T-\delta,T]},  
\end{equation}
where \( C:[0,\infty)\to[0,\infty) \) is a modulus of continuity that can be explicitly written, meaning that it is a continuous increasing function and satisfies \( C(0) = 0 \). To simplify the notation, \( C(R) \) may vary from line to line. By applying Proposition 5.5 in \cite{song2017multi} and noting that \( D_s Y(t,s) = Z(t,s) \) and \( D_s Y_t(t,s) = Z_t(t,s) \) as shown in Proposition 5.3 of \cite{Karoui1997}, we have
\begin{equation} \label{EstimateDeltaYZ}
    \begin{split}
        & \|\Delta(Y,Z)\|^2_{\mathcal{M}[T-\delta,T]}\leq 2\|\Delta(Y,Z)\|^2_{\mathcal{N}[T-\delta,T]} \\
        \leq & 2\left\{\sup\limits_{0\leq\theta\leq T}\sup\limits_{T-\delta\leq t\leq s\leq T}\mathbb{E}\left[\left.\int^T_s\left|\mathcal{G}_{\theta,\tau}(\overline{Y}^1,\overline{Z}^1)-\mathcal{G}_{\theta,\tau}(\overline{Y}^2,\overline{Z}^2)\right|d\tau\right|\mathcal{F}_s\right]\right\}^2,  
    \end{split}
\end{equation}
where $\mathcal{G}_{\theta,\tau}(\overline{Y}^1,\overline{Z}^1)$ is the generator of the BSDE system \eqref{Mapping} after substituting $(\overline{Y},\overline{Z})$ with $(\overline{Y}^1,\overline{Z}^1)$. Then, we briefly outline how to deal with the most challenging terms on the right-hand side of \eqref{EstimateDeltaYZ}:
\begin{equation*}
    \begin{split}
        & \left\{\sup\limits_{T-\delta\leq t\leq s\leq T}\mathbb{E}\left[\left.\int^T_s\left|\Delta \overline{Z}(t,\tau)\cdot D_\theta \overline{Z}^1(t,\tau)\cdot \overline{Z}^1_t(t,\tau)\right|d\tau\right|\mathcal{F}_s\right]\right\}^2 \\
        \leq & R \cdot\sup\limits_{T-\delta\leq t\leq s\leq T}\left\{\mathbb{E}\left[\left.\int^T_s\left|D_\theta \overline{Z}^1(t,\tau)\right|d\tau\right|\mathcal{F}_s\right]\right\}^2 \cdot \|\Delta(\overline{Y},\overline{Z})\|^2_{\mathcal{M}[T-\delta,T]} \\
        \leq & R \cdot\sup\limits_{T-\delta\leq t\leq s\leq T}\mathbb{E}\left[\left.\int^T_sd\tau\right|\mathcal{F}_s\right]\mathbb{E}\left[\left.\int^T_s\left|D_\theta \overline{Z}^1(t,\tau)\right|^2d\tau\right|\mathcal{F}_s\right] \cdot \|\Delta(\overline{Y},\overline{Z})\|^2_{\mathcal{M}[T-\delta,T]} \\
        \leq & R^2 \cdot \delta \cdot \|\Delta(\overline{Y},\overline{Z})\|^2_{\mathcal{M}[T-\delta,T]}
    \end{split}
\end{equation*}
and 
\begin{equation*}
    \begin{split}
        & \left\{\sup\limits_{T-\delta\leq t\leq s\leq T}\mathbb{E}\left[\left.\int^T_s\left|\int^\tau_t\Delta \overline{Z}_t(\alpha,\tau)d\alpha\cdot \int^\tau_tD_\theta\overline{Z}^1_t(\alpha,\tau)d\alpha\cdot \overline{Z}^1_t(t,\tau)\right|d\tau\right|\mathcal{F}_s\right]\right\}^2 \\
        \leq & R\cdot \left\{\sup\limits_{T-\delta\leq t\leq s\leq T}\mathbb{E}\left[\left.\int^T_s\left|\int^\tau_t\Delta \overline{Z}_t(\alpha,\tau)d\alpha\cdot \int^\tau_tD_\theta\overline{Z}^1_t(\alpha,\tau)d\alpha\right|d\tau\right|\mathcal{F}_s\right]\right\}^2 \\
        \leq & R\cdot \sup\limits_{T-\delta\leq t\leq s\leq T}\mathbb{E}\left[\left.\int^T_s\left|\int^\tau_t\Delta \overline{Z}_t(\alpha,\tau)d\alpha \right|^2d\tau\right|\mathcal{F}_s\right] \\
        & \qquad\qquad\qquad\qquad 
        \cdot \sup\limits_{T-\delta\leq t\leq s\leq T}\mathbb{E}\left[\left.\int^T_s\left|\int^\tau_tD_\theta\overline{Z}^1_t(\alpha,\tau)d\alpha\right|^2 d\tau\right|\mathcal{F}_s\right] \\
        \leq & R\cdot \sup\limits_{T-\delta\leq t\leq s\leq T}\mathbb{E}\left[\left.\int^T_s(\tau-t)\int^\tau_t\left|\Delta \overline{Z}_t(\alpha,\tau)\right|^2d\alpha d\tau\right|\mathcal{F}_s\right] \\
        & \qquad\qquad\qquad\qquad 
        \cdot \sup\limits_{T-\delta\leq t\leq s\leq T}\mathbb{E}\left[\left.\int^T_s(\tau-t)\int^\tau_t\left|D_\theta\overline{Z}^1_t(\alpha,\tau)\right|^2d\alpha d\tau\right|\mathcal{F}_s\right] \\
        \leq & R \cdot\delta^2\cdot \sup\limits_{T-\delta\leq t\leq s\leq T}\mathbb{E}\left[\left.\int^T_s\int^\tau_t\left|\Delta \overline{Z}_t(\alpha,\tau)\right|^2d\alpha d\tau\right|\mathcal{F}_s\right] \\
        & \qquad\qquad\qquad\qquad 
        \cdot \sup\limits_{T-\delta\leq t\leq s\leq T}\mathbb{E}\left[\left.\int^T_s\int^\tau_t\left|D_\theta\overline{Z}^1_t(\alpha,\tau)\right|^2d\alpha d\tau\right|\mathcal{F}_s\right] \\
        \leq & R\cdot\delta^2\cdot \sup\limits_{T-\delta\leq t\leq s\leq T}\mathbb{E}\left[\left.\int^T_s\int^\tau_{T-\delta}\left|\Delta \overline{Z}_t(\alpha,\tau)\right|^2d\alpha d\tau\right|\mathcal{F}_s\right] \\
        & \qquad\qquad\qquad\qquad 
        \cdot \sup\limits_{T-\delta\leq t\leq s\leq T}\mathbb{E}\left[\left.\int^T_s\int^\tau_{T-\delta}\left|D_\theta\overline{Z}^1_t(\alpha,\tau)\right|^2d\alpha d\tau\right|\mathcal{F}_s\right].
    \end{split}
\end{equation*}

Next, let us consider 
\begin{equation*}
    \begin{split}
        & \sup\limits_{T-\delta\leq t\leq s\leq T}\mathbb{E}\left[\left.\int^T_s\int^\tau_{T-\delta}\left|D_\theta \overline{Z}^1_t(\alpha,\tau)\right|^2d\alpha d\tau\right|\mathcal{F}_s\right] \\
        \leq & \sup\limits_{T-\delta\leq t\leq s\leq T}\mathbb{E}\left[\left.\int^T_s\int^T_\alpha\left|D_\theta \overline{Z}^1_t(\alpha,\tau)\right|^2 d\tau d\alpha\right|\mathcal{F}_s\right] \\
        & \qquad\qquad\qquad\qquad 
        +\sup\limits_{T-\delta\leq t\leq s\leq T}\mathbb{E}\left[\left.\int^s_{T-\delta}\int^T_s\left|D_\theta \overline{Z}^1_t(\alpha,\tau)\right|^2 d\tau d\alpha\right|\mathcal{F}_s\right] \\
        \leq & \sup\limits_{T-\delta\leq t\leq s\leq T}\mathbb{E}\left[\left.\int^T_s\mathbb{E}\left[\left.\int^T_\alpha\left|D_\theta \overline{Z}^1_t(\alpha,\tau)\right|^2 d\tau\right|\mathcal{F}_\alpha\right] d\alpha\right|\mathcal{F}_s\right] \\
        & \qquad\qquad\qquad 
        +\sup\limits_{T-\delta\leq t\leq s\leq T}\int^s_{T-\delta}\mathbb{E}\left[\left.\int^T_s\left|D_\theta \overline{Z}^1_t(\alpha,\tau)\right|^2 d\tau \right|\mathcal{F}_s\right]d\alpha \\
        \leq & R\cdot\sup\limits_{T-\delta\leq t\leq s\leq T}\mathbb{E}\left[\left.\int^T_s1 d\alpha\right|\mathcal{F}_s\right]+R\cdot\sup\limits_{T-\delta\leq t\leq s\leq T}\int^s_{T-\delta}1d\alpha \leq 2\delta R. 
    \end{split}
\end{equation*}

Similarly, one has 
\begin{equation*}
    \sup\limits_{T-\delta\leq t\leq s\leq T}\mathbb{E}\left[\left.\int^T_s\int^\tau_{T-\delta}\left|\Delta\overline{Z}_t(\alpha,\tau)\right|^2d\alpha d\tau\right|\mathcal{F}_s\right]\leq 2\delta^2 \|\Delta(\overline{Y},\overline{Z})\|^2_{\mathcal{M}[T-\delta,T]}. 
\end{equation*}

Consequently, 
\begin{equation*}
    \begin{split}
        &\left\{\sup\limits_{T-\delta\leq t\leq s\leq T}\mathbb{E}\left[\left.\int^T_s\left|\int^\tau_t\Delta D_\tau\overline{Y}_t(\alpha,\tau)d\alpha\cdot \int^\tau_tD_\theta\overline{Z}^1_t(\alpha,\tau)d\alpha\cdot D_\tau\overline{Y}^1_t(t,\tau)\right|d\tau\right|\mathcal{F}_s\right]\right\}^2 \\
        \\[-13pt]
        \leq & 4R^2\delta^5\|\Delta(\overline{Y},\overline{Z})\|^2_{\mathcal{M}[T-\delta,T]}. 
    \end{split}
\end{equation*}

Through a rather lengthy but straightforward verification, one can establish the desired inequality \eqref{Contraction}. The established inequality \eqref{Contraction} implies three important facts: 
\begin{enumerate}
    \item By letting $(\overline{Y}^2,\overline{Z}^2)=(0,0)$, one can find that 
    \begin{equation*}
    \left\{\sup\limits_{T-\delta\leq t\leq s\leq T}\mathbb{E}\left[\left.\int^T_s\left|\mathcal{G}_{\theta,\tau}(\overline{Y}^1,\overline{Z}^1)-\mathcal{G}_{\theta,\tau}(0,0)\right|d\tau\right|\mathcal{F}_s\right]\right\}^2\leq C(R)\delta\|(\overline{Y}^1,\overline{Z}^1)\|^2_{\mathcal{M}[T-\delta,T]}, 
\end{equation*}
where $\mathcal{G}_{\theta,\tau}(0,0)$ is actually the value of $(g,g_t,D_\theta g,D_\theta g_t)$ evaluated at $(t,s,0,0,0,0)$. Consequently, one has 
\begin{equation*}
    \begin{split}
        & \left\{\sup\limits_{T-\delta\leq t\leq s\leq T}\mathbb{E}\left[\left.\int^T_s\left|\mathcal{G}_{\theta,\tau}(\overline{Y}^1,\overline{Z}^1)\right|d\tau\right|\mathcal{F}_s\right]\right\}^2 \\
        \leq & C(R)\delta\|(\overline{Y}^1,\overline{Z}^1)\|^2_{\mathcal{M}[T-\delta,T]}+C\left\{\sup\limits_{T-\delta\leq t\leq s\leq T}\mathbb{E}\left[\left.\int^T_s\left|\mathcal{G}_{\theta,\tau}(0,0)\right|d\tau\right|\mathcal{F}_s\right]\right\}^2 <\infty, 
    \end{split}
\end{equation*}
which directly indicates the mapping $\Gamma((\overline{Y}^1,\overline{Z}^1))=(Y^1,Z^1)$ is well-defined over the space $\mathcal{N}[T-\delta,T]$; see also the discussion in \cite{song2017multi}.
\item If $\delta$ and $R$ satisfy $C(R)\delta\leq \frac{1}{4}$, then $\Gamma$ is a $\frac{1}{4}$-contraction.
\item Assuming that \(\delta\) and \(R\) are balanced in such a way that \(\Gamma\) becomes a \(\frac{1}{4}\)-contraction, for every $(Y,Z)\in B_{\delta,T}(R)$, 
\begin{equation} \label{Selfmapping}
    \begin{split}
        \|\Gamma(Y,Z)\|^2_{\mathcal{M}[T-\delta,T]} & \leq 2\|\Gamma(Y,Z)-\Gamma(0,0)\|^2_{\mathcal{M}[T-\delta,T]}+2\|\Gamma(0,0)\|^2_{\mathcal{M}[T-\delta,T]} \\
        & \leq \frac{R}{2}+2\|\Gamma(0,0)\|^2_{\mathcal{M}[T-\delta,T]},
    \end{split}
\end{equation}
where $\Gamma(0,0)$ is the solution of 
\begin{equation*}
    \begin{cases}
        \begin{aligned}
            dY(t,s) & ~=~ g(t,s,0,0,0,0)ds+Z(t,s)dB(s), \\
        
            Y(t,T) & ~=~ \xi(t), \quad T-\delta\leq t \leq s\leq T,  
        \end{aligned}
    \end{cases}
\end{equation*} 
which is estimated by  
\begin{equation} \label{Mappingat0}
    \begin{cases}
        \begin{aligned}
            dY(t,s) & ~=~ \overline{g}(t,s,0,0,0,0)ds+Z(t,s)dB(s), \\

            dY_t(t,s) & ~=~ \overline{g}_t(t,s,0,0,0,0)ds+Z_t(t,s)dB(s), \\
            
            dD_\theta Y(t,s) & ~=~D_\theta \overline{g}(t,s,0,0,0,0)ds+D_\theta Z(t,s)dB(s), \\

            dD_\theta Y_t(t,s) & ~=~ D_\theta \overline{g}_t (t,s,0,0,0,0)ds +D_\theta Z_t(t,s)dB(s)
        \end{aligned}  
    \end{cases}
\end{equation}
with $(Y,Y_t,D_\theta Y,D_\theta Y_t)(t,T) = (\xi,\xi_t,D_\theta\xi,D_\theta \xi_t)(t)$ for $T-\delta\leq t\leq s\leq T$ and $0\leq \theta\leq s\leq T$. Moreover, $D_\theta Y(t,s)=D_\theta Z(t,s)=D_\theta Y_t(t,s)=D_\theta Z_t(t,s)=0$ for $s\leq \theta\leq T$. Consequently, the inequality \eqref{Selfmapping} shows that $\Gamma$ is a self-mapping defined in $B_{\delta,T}(R)$ only if $4\|\Gamma(0,0)\|^2_{[T-\delta,T]}\leq R$. 
\end{enumerate}

From the analyses above, we observe that for sufficiently large \( R \) and small enough \( \delta \), Banach's fixed-point theorem guarantees the existence of a unique fixed point for \( \Gamma \) within the closed ball \( B_{\delta,T}(R) \), which serves as a solution to the BSVIE \eqref{GeneralCase} in \( \Delta[T-\delta, T] \). 

\vspace{0.3cm}

\noindent\textbf{Step 2.} Now, we are to show that the solution found within the closed ball \(B_{\delta,T}(R)\) (in Step 1) is unique in the entire space $\mathcal{M}[T-\delta, T]$. It can be done with standard arguments below.

If \eqref{GeneralCase} has two adapted solutions $(Y^1,Z^1)$ and $(Y^2,Z^2)\in\mathcal{M}[T-\delta, T]$, set $\overline{s}=\sup\{s\in[T-\delta,T]:Y^1(t,s)=Y^2(t,s) \text{~~in~~} T-\delta\leq t\leq s\leq T\}$. If $\overline{s}=T-\delta$, a straightforward application of the inequality (11) from \cite{song2017multi} implies that \((Y^1,Z^1)(t,s) = (Y^2,Z^2)(t,s)\) for \(T - \delta \leq t \leq s \leq T\). The proof is finished. If $\overline{s}>T-\delta$, we consider the BSVIE with a updated terminal datum: 
\begin{equation} \label{Updated BSVIE}
    \begin{cases}
        \begin{aligned}
            dY(t,s) & ~=~ g(t,s,Y(t,s),Z(t,s),Y(s,s),Z(s,s))ds+Z(t,s)dB(s), \\
        
            Y(t,T) & ~=~ \overline{\xi}(t), \quad 0\leq t \leq s\leq \overline{s}, 
        \end{aligned}  
    \end{cases}
\end{equation} 
where $\overline{\xi}(t)=Y^1(t,\overline{s})=Y^2(t,\overline{s})$. Remarkably, since \((Y,Z)(s,s)\) is well-defined on \([\overline{s},T]\), the two solutions \((Y^1,Z^1)(t,s)\) and \((Y^2,Z^2)(t,s)\) naturally extend to \((t,s) \in ([0,\overline{s}] \times [\overline{s},T]) \cup \Delta[\overline{s},T]\) from \(\Delta[\overline{s},T]\); see Step 2 of the proof of Theorem 3.1 in \cite{Wang2020}. As a result, \(\overline{\xi}(t)\) is well-defined for \(t \in [0,\overline{s}]\). The proof in Step 1 above shows that \eqref{Updated BSVIE} has a unique solution in the set $B_{\delta^\prime,\overline{s}}(R^\prime):=\big\{(Y,Z)(t,s)\in\mathcal{M}[\overline{s}-\delta^\prime, \overline{s}]:\|(Y,Z)\|^2_{\mathcal{M}[\overline{s}-\delta^\prime, \overline{s}]}\leq R^\prime\big\}$, provided that $R^\prime$ is sufficiently large and $\delta^\prime$ is small enough. Taking $R^\prime$ larger than both $\|(Y^1,Z^1)(t,s)\|^2_{[T-\delta,\overline{s}]}$ and $\|(Y^2,Z^2)(t,s)\|^2_{[T-\delta,\overline{s}]}$, we obtain $Y^1(t,s)=Y^2(t,s)$ in $[\overline{s}-\delta^\prime,\overline{s}]$, which contradicts the definition of $\overline{s}$. Therefore, with \(\overline{s} = T - \delta\), we have \(Y^1(t,s) = Y^2(t,s)\) in $\Delta[T-\delta,T]$, which in turn leads to \(Z^1(t,s) = Z^2(t,s)\). 

\vspace{0.3cm}

\noindent\textbf{Step 3.} Finally, we complete the proof by demonstrating that the local well-posedness of solutions to \eqref{GeneralCase} on \( \Delta[T - \delta, T] \), established in Steps 1 and 2, can be extended to any arbitrarily large time interval \( \Delta[0, T] \). Specifically, this approach redefines \( T - \delta \) as the new terminal time and uses \( Y(t, T - \delta) \) for \( t \in [0, T - \delta] \) as the updated terminal condition, as described in Step 2. It allows the solution to be extended over a larger time interval; see also Step 2 of the proof of Theorem 3.1 in \cite{Wang2020}. This process can be repeated indefinitely until a potential explosion point is encountered. To ensure that \( \Gamma \) retains its contractive property in this process, it is crucial to maintain a proper balance between \( R \) and \( \delta \) in \eqref{Contraction}, noted the facilitation of the fixed-point arguments in Step 1. If \( R \) is allowed to grow indefinitely during these repeated extensions, the corresponding interval length would eventually shrink toward zero and halt the extension process. Therefore, to ensure each extension associated with an interval of strictly positive period, it is necessary to avoid any instance where \( R \) encounters an explosion.

To this end, we first provide an a-priori estimate for solutions of the BSVIE \eqref{GeneralCase}, showing that any solution of \eqref{GeneralCase} in \( \mathcal{M}[0,T] \) is bounded by an absolute constant \( C > 0 \) that depends only on the intrinsic system parameters \( d \), \( n \), \( L \), \( T \), \( \xi \), and \( g(t,s,0,0,0,0) \). We assume there exists a solution \((Y, Z) \in \mathcal{M}[0, T]\) for \eqref{GeneralCase} in \(\Delta[0, T]\). By directly applying Propositions 5.1 and 5.5 of \cite{song2017multi} to the BSVIE system \eqref{SysYYt} in $\Delta[T-\delta,T]$ and choosing a small enough $\delta>0$, we obtain the following results:
\begin{equation*}
    \begin{split}
        \|(Y,Z,Y_t,Z_t)\|^2_{\mathcal{N}[T-\delta,T]} \leq & C(\delta)\|(Y,Z,Y_t,Z_t)\|^2_{\mathcal{N}[T-\delta,T]}+C\left(\|\xi\|^2_{\Lambda^\infty_{\mathcal{F}_T}([T-\delta,T])}+\|g_0\|^2_{\Theta^{\textsc{BMO}}_{\mathbb{F}}(\Delta[T-\delta,T])}\right) \\
        \leq & \frac{1}{2}\|(Y,Z,Y_t,Z_t)\|^2_{\mathcal{N}[T-\delta,T]}+C\left(\|\xi\|^2_{\Lambda^\infty_{\mathcal{F}_T}([0,T])}+\|g_0\|^2_{\Theta^{\textsc{BMO}}_{\mathbb{F}}(\Delta[0,T])}\right), 
    \end{split}
\end{equation*}
where $(Y,Z,Y_t,Z_t)$ is understood as a vector of the product space between two spaces of $\mathcal{N}[T-\delta,T]$, $g_0=g(t,s,0,0,0,0)$, \( C \) is a positive constant defined as \( C := C(n, k, L, T) > 0 \), and \( C(\delta):=C(n, k, L, T, \delta) \) denotes a certain modulus of continuity. These generic constants could vary from line to line. Consequently, by repeatedly applying the inequality above over sufficiently small sub-intervals, we obtain a finite constant \( N > 0 \) that depends only on the intrinsic parameters \( d \), \( n \), \( L \), \( T \), \( \|\xi\|^2_{\Lambda^\infty_{\mathcal{F}_T}([0,T])} \), and \( \|g_0\|^2_{\Theta^{\textsc{BMO}}_{\mathbb{F}}(\Delta[0,T])} \), such that
\begin{equation} \label{BoundednessYYt}
    \|(Y,Z,Y_t,Z_t)\|^2_{\mathcal{N}[0,T]}\leq N. 
\end{equation}
It is important to note that this process requires only a finite number of deterministic steps, differing from the potential explosion points discussed earlier, as \( C(\delta) \) is a finite quantity that depends solely on the system parameters \( (n, k, L, T) \).

Similarly, applying Propositions 5.1 and 5.5 from \cite{song2017multi} to the first BSDE of \eqref{MalliavinYYt} yields
\begin{equation*}
    %
    \|(D_\theta Y,D_\theta Z)\|^2_{\mathcal{N}[T-\delta,T]} \leq C(\delta)\|(D_\theta Y_t,D_\theta Z_t)\|^2_{\mathcal{N}[T-\delta,T]}+C\left(\|\xi\|^2_{\Lambda^\infty_{\mathcal{F}_T}([0,T])}+\|g_0\|^2_{\Theta^{\textsc{BMO}}_{\mathbb{F}}(\Delta[0,T])}\right). 
\end{equation*}
Subsequently, by choosing a small enough $\delta>0$, the second BSDE of \eqref{MalliavinYYt} tells us that 
\begin{equation*}
    \begin{split}
         & \|(D_\theta(Y,Z,Y_t,Z_t))\|^2_{\mathcal{N}[T-\delta,T]} \\
         \leq & C(\delta,N)\cdot\|D_\theta(Y,Z,Y_t,Z_t)\|^2_{\mathcal{N}[T-\delta,T]}+C(N)\cdot\left(\|\xi\|^2_{\Lambda^\infty_{\mathcal{F}_T}([0,T])}+\|g_0\|^2_{\Theta^{\textsc{BMO}}_{\mathbb{F}}(\Delta[0,T])}\right) \\ 
         \leq & \frac{1}{2}\|D_\theta(Y,Z,Y_t,Z_t)\|^2_{\mathcal{N}[T-\delta,T]}+C(N)\left(\|\xi\|^2_{\Lambda^\infty_{\mathcal{F}_T}([0,T])}+\|g_0\|^2_{\Theta^{\textsc{BMO}}_{\mathbb{F}}(\Delta[0,T])}\right). 
    \end{split} 
\end{equation*}
By the same arguments, repeatedly applying the inequality above over sufficiently small intervals leads to the conclusion that there exists a finite constant \( M > 0 \) that depends only on the intrinsic parameters \( d \), \( n \), \( L \), \( T \), \( \|\xi\|^2_{\Lambda^\infty_{\mathcal{F}_T}([0,T])} \), and \( \|g_0\|^2_{\Theta^{\textsc{BMO}}_{\mathbb{F}}(\Delta[0,T])} \), such that
\begin{equation} \label{BoundednessMYYt}
    \|D_\theta(Y,Z,Y_t,Z_t)\|^2_{\mathcal{N}[0,T]}\leq M. 
\end{equation}

  Before reaching \eqref{BoundednessMYYt}, the derivation only made use of the \(\|\cdot\|_{\mathcal{N}}\)-boundedness of \(Z\) in \(\Theta^{\textsc{BMO}}_{\mathbb{F}}(\Delta[0,T])\) as shown in \eqref{BoundednessYYt}. At this point, we only know that \(Z\) is integrable and do not assume that it is a bounded process. More specifically, we only used the boundedness $N$ of \(\|Z\|_{\Theta^{\textsc{BMO}}_{\mathbb{F}}(\Delta[0,T])}\) to establish \eqref{BoundednessMYYt} without relying on the undetermined \(\|Z\|_{\Theta^{\textsc{BMO}}_{\mathbb{F},c}}\)-norm. The boundedness of \(\|Z\|_{\Theta^{\textsc{BMO}}_{\mathbb{F},c}(\Delta[0,T])}\) is, in fact, the result we desire to prove, and thus we cannot presume it prior to showing it. Next, by recognizing that \(Z(t,s) = D_s Y(t,s)\) and \(Z_t(t,s) = D_s Y_t(t,s)\), the results in \eqref{BoundednessYYt} and \eqref{BoundednessMYYt} altogether imply that \(Z(t,s)\) is not only in \(\Omega^{\textsc{BMO}}_{\mathbb{F}}(\Delta[0,T])\) but also belongs to \(\Omega^{\textsc{BMO}}_{\mathbb{F},c}(\Delta[0,T])\).

Consequently, by combining \eqref{BoundednessYYt} and \eqref{BoundednessMYYt}, we deduce the existence of a constant \(C\), dependent only on the system parameters \(d\), \(n\), \(L\), \(T\), \(\|\xi\|^2_{\Lambda^\infty_{\mathcal{F}_T}([0,T])}\), and \(\|g_0\|^2_{\Theta^{\textsc{BMO}}_{\mathbb{F}}(\Delta[0,T])}\), which controls the behavior of \((Y, Z)(t,s)\) in the space \(\mathcal{M}[0,T]\). Specifically, we have \(\|(Y, Z)\|_{\mathcal{M}[0,T]} \leq C\). This allows the process of extending the solution from a local interval to a larger one to be repeated indefinitely, leading to the construction of a global solution to \eqref{GeneralCase} over any arbitrarily large time interval \( \Delta[0, T] \).
\end{proof}

The proof of Theorem \ref{GeneralWellp} presents a novel approach that deals with non-Lipschitz structure of the coupled system \eqref{SysYYt}-\eqref{MalliavinYYt}, while one can find similar attempts in the PDE literature, such as \cite{Lunardi1995, lunardi2002nonlinear, LeiSDG, lei2023well}. In contrast, the standard approach for studying nonlinear BSDEs uses fixed-point arguments to establish the (local) existence of the solutions while for the uniqueness, shows the unique trivial solution to the linear BSDEs with bounded coefficients (under Lipschitz assumptions) and a zero terminal condition that drive the difference between two potential solutions. We stress that the standard approach cannot be applied to show the existence result for the non-Lipschitz structure. Though, for the uniqueness, following the standard approach is possible. To see this, one may similarly consider two solutions of \eqref{GeneralCase} in $\mathcal{M}[0,T]$ whose difference satisfies a linear \textit{nonlocal} BSDE system with bounded coefficients. Note that the boundedness of its coefficients arises from the properties of the \( \mathcal{M} \)-space rather than by making assumptions on the generator as in \cite{hernandez2023me, Hernandez2021a}. Consequently, the well-posedness results in \cite{lei2024MV} imply that the two solutions are identical. 

\begin{corollary} \label{Cor:Zrep}
    The martingale integrand process $\{Z(t,s)\}_{0\leq t\leq s\leq T}$ of the unique adapted solution in $\mathcal{M}[0,T]$ of \eqref{GeneralCase} can be expressed in terms of the trace of the Malliavin derivative of the solution process $\{Y(t,s)\}_{0\leq t\leq s\leq T}$, i.e., $Z(t,s)=D_sY(t,s)$ for $0\leq t\leq s\leq T$. 
\end{corollary}

Corollary \ref{Cor:Zrep} directly follows from the study of Malliavin differentiability of BSDEs, as discussed in \cite{Pardoux1992, Karoui1997, hu2011malliavin, yong2008well, Wang2020}. Interestingly, the trace process \( \{D_s Y(t,s)\}_{0 \leq t \leq s \leq T} \) resembles a diagonal process. However, what distinguishes it from the diagonal process \( \{Z(s,s)\}_{0 \leq s \leq T} \) is that \(D_s Y(t,s) = \lim_{\tau \to s^+} D_s Y(t,\tau)\) is always well-defined because \(D_\theta Y(t,s)\) is continuous in the \(s\)-direction. Moreover, from the expression \( Z(t,s) = D_s Y(t,s) \), we can see that imposing the continuity of \( Z(t,s) \) in \(s\) is more demanding than the continuity of \( D_\theta Y(t,s) \) in \(s\) as the former would also require the continuity of \( D_\theta Y(t,s) \) in \(\theta\). This observation motivates our further investigation.

Next, we aim to show that the mapping \(s \mapsto (Y,Z)(t,s)\) is continuous by imposing additional continuity assumptions (A4) and (A5) below to the Malliavin derivatives of the generator $g$ and terminal datum $\xi$ of \eqref{General BSVIE} on top of the original assumptions (A1)–(A3). Such a result not only serves as an effective tool for estimating \( \mathbb{E}[\sup_s |Z(t,s)|^2] \), which is crucial for estimating the rate of convergence in numerical schemes as discussed in \cite{hu2011malliavin}, but also ensures that the intuitive re-formulation of the diagonal processes \(Z(s, s) := Z(s, \tau)|_{\tau=s} = D_\tau Y(s, \tau)|_{\tau=s} = \lim_{\tau\to s^+}D_\tau Y(s,\tau)\) is well-defined. This perspective of handling the diagonal processes, particularly for \( \{Z(s,s)\}_{0 \leq s \leq T} \), significantly differs from the existing approaches in the BSVIE literature \cite{Hamaguchi2021, hernandez2023me, Hernandez2021a, Wang2021, lei2023nonlocal, LeiSDG, lei2023well}. Moreover, the path regularity of solutions to \( Z(s,s) \)-dependent BSVIEs has been primarily studied in the context of Markovian cases \cite{Wang2021, lei2023nonlocal, LeiSDG, lei2023well}, where the terminal value and generator are functions of a forward diffusion. In contrast, we consider the general cases.

In addition to (A1)-(A3), we further impose the following conditions on the Malliavin derivatives of $(\xi,g)$ of \eqref{GeneralCase}.
\begin{enumerate}[start=4,label={\upshape(\bfseries A\arabic*).}]
    \item There exist $\beta$, $L>0$ such that for all $\theta$, $\theta^\prime\in[0,T]$ and $\eta\in\{\xi,\xi_t\}$,  
    \begin{equation*}
        \mathbb{E}\big[|D_\theta\eta-D_{\theta^\prime}\eta|^2\big]\leq L|\theta-\theta^\prime|^{1+\beta}.   
    \end{equation*}
    \item (Under assumptions (A1)-(A3)) Let $(Y,Z)$ be the unique solution in $\mathcal{M}[0,T]$ to \eqref{GeneralCase} in $\Delta[0,T]$. Moreover, there exist $\beta$, $L>0$ such that for all $0\leq\theta,\theta^\prime\leq s\leq T$ and $h\in\{g,g_t,g_Y,g_Z\}$, 
    \begin{equation*}
        \begin{split}
            & \mathbb{E}\Big[\int\nolimits_s^T|D_\theta h(t,\tau,Y(t,\tau),Z(t,\tau),Y(\tau,\tau),Z(\tau,\tau)) \\ 
            &  \qquad  
            -D_{\theta^\prime} h(t,\tau,Y(t,\tau),Z(t,\tau),Y(\tau,\tau),Z(\tau,\tau))|^2d\tau \Big] \leq L|\theta-\theta^\prime|^{1+\beta}. 
        \end{split}
    \end{equation*}
\end{enumerate}

Notably, the additional Hölder continuity conditions in (A4)–(A5) only need to hold provided the \(\mathcal{F}_0\)-information set. The following theorem then concludes our desired claims.
\begin{theorem} \label{ContinuityZ}
    Under assumptions $\mathrm{(A1)}$-$\mathrm{(A5)}$, the BSVIE \eqref{GeneralCase} admits a unique adapted solution $(Y,Z)(t,s)\in\mathcal{M}_c[0,T]$. The solution process $\{Y(t,s)\}_{0\leq t\leq s\leq T}$ and the martingale integrand process $\{Z(t,s)\}_{0\leq t\leq s\leq T}$ are both $\mathbb{F}$-progressively measurable bounded and continuous processes. Consequently, the diagonal process pair of $\{(Y,Z)(s,s)\}_{0\leq s\leq T}$ represented by $(Y,Z)(s,s)=(Y,Z)(s,\tau)|_{\tau=s}=\lim_{\tau \to s^+} (Y,Z)(s,\tau)$ is well-defined.  
\end{theorem}

\begin{proof}
    First of all, Theorem \ref{GeneralWellp} promises that the BSVIE \eqref{GeneralCase} admits a global unique adapted solution in $\mathcal{M}[0,T]$ and that the mapping $s\mapsto (D_\theta Y,D_\theta Y_t)(t,s)$ is continuous almost surely. Thanks to the representations of $Z(t,s)=D_sY(t,s)$ and $Z_t(t,s)=D_sY_t(t,s)$, we have 
    \begin{equation*}
        \begin{split}
            |Z(t,s)-Z(t,s^\prime)| & = |D_sY(t,s)-D_{s^\prime}Y(t,s^\prime)| \\
            & \leq |D_sY(t,s)-D_{s^\prime}Y(t,s)|+|D_{s^\prime}Y(t,s)-D_{s^\prime}Y(t,s^\prime)|. 
        \end{split}
    \end{equation*}
    A similar estimation also holds for \( Z_t(t,s) \). Hence, we only need to investigate the continuity of $\theta\mapsto(D_\theta Y,D_\theta Y_t)(t,s)$. By following the arguments of Theorem 2.6 in \cite{hu2011malliavin}, it is clear from \eqref{MalliavinYYt} that the BSVIE system satisfied by the difference between $D_\theta(Y,Z,Y_t,Z_t)$ and $D_{\theta^\prime}(Y,Z,Y_t,Z_t)$ reduces to a linear nonlocal BSDE (BSVIE) system with bounded coefficients provided that $(Y,Z)\in\mathcal{M}[0,T]$; see \cite{lei2024MV}. Utilizing the continuous dependence of solutions on system parameters \((g, \xi)\) from Lemma 2.2 in \cite{hu2011malliavin}, alongside with the arguments in \cite{lei2024MV} and the proof of Theorem \ref{GeneralWellp}, it can be shown that 
    \begin{equation*}
        \mathbb{E}\big[|D_\theta Y(t,s)-D_{\theta^\prime}Y(t,s)|^2\big]+\mathbb{E}\big[|D_\theta Y_t(t,s)-D_{\theta^\prime}Y_t(t,s)|^2\big]\leq L|\theta-\theta^\prime|^{1+\beta}.
    \end{equation*}
    As a result, by Kolmogorov's continuity theorem, there exists a Hölder-\(\gamma\) continuous modification of \(\theta \mapsto (D_\theta Y, D_\theta Y_t)(t, s)\) for any \(\gamma \in (0, \frac{\beta}{2})\). The proof is completed.
\end{proof}

From Theorem \ref{GeneralWellp} and Theorem \ref{ContinuityZ}, it is evident that Malliavin calculus is a crucial tool for analyzing the general BSVIE \eqref{GeneralCase}. This approach allows us to study the pointwise behavior of \( s \mapsto (Y, Z)(t, s) \), going beyond integrability. It offers two key advantages: (1) a clear definition of the diagonal process \( (Y, Z)(s, s) = (Y, Z)(s, \tau) |_{\tau = s} \) and (2) under minimal assumptions on the nonlinearity of \(g\), it ensures the existence and uniqueness of the solutions over an arbitrarily large time interval. Moreover, in a similar spirit of \cite{hu2011malliavin}, we can even improve the regularity of \( s \mapsto (Y, Z)(t, s) \) to H\"{o}lder continuity.



\subsection{Extensions to Non-Uniformly Lipschitz Cases}
Previously, we imposed uniform Lipschitz conditions on the generator and its first- and second-order derivatives to simplify the setup and highlight our methodology. In this subsection, we relax these assumptions and extend the previous well-posedness results to a non-uniformly Lipschitz setting.

To see the motivation behind the extension, we examine a simple yet common case: a linear BSVIE with stochastic Lipschitz coefficients (where all components are one-dimensional):
\begin{equation} \label{NonuLBSVIE example}
    \begin{cases}
        \begin{aligned}
            dY(t,s) & ~=~ B(s)\cdot(Y(t,s)+Z(t,s)+Y(s,s)+Z(s,s))ds+Z(t,s)dB(s), \\
        
            Y(t,T) & ~=~ \xi(t), \quad 0\leq t \leq s\leq T, 
        \end{aligned}  
    \end{cases}
\end{equation} 
where the generator is Lipschitz with respect to any of its arguments with a Lipschitz constant \( B(s) \) that could vary with trajectory \( \omega \in \Omega \) and may even become unbounded. It is easy to see that the BSVIE \eqref{NonuLBSVIE example} does not satisfy the previous assumptions in Tables \ref{tab:driver1}-\ref{tab:driver2} and thus for it, we need a relaxation of those conditions to establish its well-posedness in a more general setting. Hence, the extension is both theoretically significant and practically relevant, particularly for dynamic mean-variance asset allocation in incomplete markets, as seen in Section \ref{Sec: TICAPP}, where the non-uniformly random coefficients represent the instantaneous Sharpe ratio.

Now, we appropriately relax the Lipschitz and bounded conditions in (A2). The new set of conditions will be collectively referred to as (A2'). 
\begin{enumerate}[start=2,label={\upshape(\bfseries A\arabic*').}]
    \item For any $(t,s)\in\Delta[0,T]$, and any $(y^1,z^1,\overline{y}^1,\overline{z}^1)$ and $(y^2,z^2,\overline{y}^2,\overline{z}^2)\in\mathbb{R}^{2k}\times\mathbb{R}^{2(k\times n)}$, $F_\theta$ refers to the functions with $\surd_p$ ($p=2$ or $4$) in Tables \ref{tab:driver3}-\ref{tab:driver4} and
    \begin{enumerate}[label=(\alph*)]
        \item Non-uniformly Lipschitz condition:
        \begin{equation*}
        \begin{split}
            &|F_\theta(t,s,y^1,z^1,\overline{y}^1,\overline{z}^1)-F_\theta(t,s,y^2,z^2,\overline{y}^2,\overline{z}^2)| \\
            \\[-18pt]
            & \qquad \leq K_\theta(t,s,\omega)\big(|y^1-y^2|+|z^1-z^2|+|\overline{y}^1-\overline{y}^2|+|\overline{z}^1-\overline{z}^2|\big),
        \end{split}
        \end{equation*} 
        where $\{K_\theta(t,s,\cdot)\}_{0\leq t,\theta\leq s\leq T}$ is an $\mathbb{R}^+$-valued $\mathbb{F}$-adapted process satisfying
        \begin{equation*}
            \sup_{0\leq\theta\leq T}\sup_{0\leq t\leq s\leq T}\Big\{\Big\|\mathbb{E}\Big[\int\nolimits_s^T|K_\theta(t,\tau,\cdot)|^pd\tau\Big|\mathcal{F}_s\Big]\Big\|_\infty\Big\}<\infty. 
        \end{equation*}
    \item Potential boundedness: 
    \begin{equation*}
        \sup_{0\leq\theta\leq T}\sup_{0\leq t\leq s\leq T}\Big\|\mathbb{E}\Big[\int\nolimits_s^T|F_\theta(t,\tau,0,0,0,0)|^pd\tau\Big|\mathcal{F}_s\Big]\Big\|_\infty<\infty.
    \end{equation*}
    \end{enumerate}
\end{enumerate}


\begin{table}[!ht] 
	\centering
    \begin{tabular}{c| c c c c c c c c c}
       \hline
       $\alpha$ & $\varnothing$ & $t$ & $D_\theta$ & $s$ & $y$ & $z$ & $\overline{y}$ & $\overline{z}$  \\ 
      \hline 
      $g_\alpha$ & $\surd_2$ & $\surd_2$ & $\surd_2$ & & $\surd_2$ & $\surd_4$ & $\surd_2$ & $\surd_4$ \\ 
      \hline 
    \end{tabular}
	\caption{First-order derivatives of $g$ and $g$ itself ($g_\varnothing=g$) required to be $p$-integrable}
	\label{tab:driver3}
\end{table} 

\begin{table}[!ht] 
  \centering
  \begin{tabular}{c| c c c c c c c c c}
  	\hline
  	\diagbox{$\alpha$}{$g_{\alpha\beta}$}{$\beta$} & $t$ & $D_\theta$ & $s$ & $y$ & $z$ & $\overline{y}$ & $\overline{z}$ \\ 
  	\hline 
  	$t$ &  & $\surd_2$ &  & $\surd_2$ & 
        $\surd_4$ & $\surd_2$ & $\surd_4$ \\
        $D_\theta$ & $\surd_2$ &  &  & $\surd_2$ & $\surd_2$ &  &  \\
  	$s$ &  &  &  &  & \\
  	$y$ & $\surd_2$ & $\surd_2$ &  & $\surd_2$ & $\surd_4$ & $\surd_2$ & $\surd_4$ \\
  	$z$ & $\surd_4$ & $\surd_2$ &  & $\surd_4$ & $\surd_4$ & $\surd_2$ & $\surd_4$ \\
  	$\overline{y}$ & $\surd_2$ &  &  & $\surd_2$ & $\surd_2$  &  & \\ 
   $\overline{z}$ & $\surd_4$ &  &  & $\surd_4$ & $\surd_4$ & &  \\
  	\hline 
  \end{tabular}
  \caption{Second-order derivatives of $g$ required to be $p$-integrable}
  \label{tab:driver4}
\end{table} 

The conditions in (A2) with \( \surd_{L} \) or \( \surd_{BL} \) in Tables \ref{tab:driver1}-\ref{tab:driver2} represent a special case where \( \{K_\theta(t,s,\cdot)\}_{0 \leq t, \theta \leq s \leq T} \) is a constant \( L \) uniformly in \( (t, s, \theta, \omega) \). Moreover, (A3) can be seen as part of (b) in (A2'). Note that (A2') ensures the existence of a maximally defined solution to the BSVIE \eqref{GeneralCase}, but not necessarily a global one. This means that extending local solutions to a larger time interval may be hindered by potential explosion points. To address this, we will strengthen the assumptions to (A2'') based on (A2'), ensuring global well-posedness while keeping (A2'') much weaker than (A2).

\begin{enumerate}[start=2,label={\upshape(\bfseries A\arabic*'').}]
\item On top of (A2'), the non-uniform Lipschitz conditions for \( g \) and \( g_t \) in Table \ref{tab:driver3} are elevated to hold for \( p = 4 \). Moreover, there exists \( K^\prime_\theta(t,s,\omega) \) such that \( F^\prime_\theta(t, \tau, y, \overline{y}, z, \overline{z}) \leq K^\prime_\theta(t,s,\omega) \) holds uniformly for all arguments and
$$
\sup_{0 \leq \theta \leq T} \sup_{0 \leq t \leq s \leq T} \big\{ \big\| \mathbb{E} \big[ \int_s^T |K^\prime_\theta(t, \tau, \cdot)|^q d\tau \mid \mathcal{F}_s \big] \big\|_\infty \big\} < \infty,
$$
where \( F^\prime_\theta \in \{ g_Y, g_{\overline{Y}} \} \) for \( q = 2 \), \( F^\prime_\theta \in \{ g_Z, g_{\overline{Z}} \} \) for \( q = 4 \), and \( F^\prime_\theta \in \{ g_{ZZ}, g_{Z\overline{Z}} \} \) for \( q = \infty \). 
\end{enumerate}


It is clear that among these three assumptions, (A2) is the strongest, (A2') is the weakest, and (A2'') is in between them. Based on these conditions, we present the following well-posedness result for solutions of non-uniformly Lipschitz BSVIEs. 

\begin{theorem} \label{ExtendedWellp}
Under assumptions (A1) and (A2'), there are $\delta>0$ and a unique maximally-defined $(Y,Z)(t,s)\in\mathcal{M}[T-\delta,T]$ satisfying \eqref{GeneralCase} in $\Delta[T-\delta,T]$. Furthermore, if assumptions (A1) and (A2'') hold, then $\delta=T$, i.e., the BSVIE \eqref{GeneralCase} admits a unique adapted solution $(Y,Z)(t,s)\in\mathcal{M}[0,T]$. 
\end{theorem}
\begin{proof}
    The proof closely follows the that of Theorem \ref{GeneralWellp}. Hence, we provide only a snapshot here rather than a full detailed proof. Let us consider two typical \( F_\theta \) functions with \( \surd_4 \) in Table \ref{tab:driver3}-\ref{tab:driver4}, such as \( g_{zz} \) and \( g_{z\overline{z}} \). These are closely related to \( \widetilde{g}_{ZZ} \) from \eqref{Mapping}, noting that \( \widetilde{g}_{ZZ} = \overline{g}_{ZZ} + \overline{g}_{Z\overline{Z}} \). 
When we construct a contraction similar to the method in Theorem \ref{GeneralWellp} to demonstrate its local well-posedness, \eqref{Mapping} indicates that we need to estimate
\begin{equation} \label{NuniformE1}
    \left\{\sup\limits_{T-\delta\leq t\leq s\leq T}\mathbb{E}\left[\left.\int^T_s\left|K_\theta(t,\tau,\omega)\cdot\Delta \overline{Z}(t,\tau)\cdot D_\theta\overline{Z}^1(t,\tau)\cdot \overline{Z}^1_t(t,\tau)\right|d\tau\right|\mathcal{F}_s\right]\right\}^2   
\end{equation}
and 
\begin{equation*}
    \left\{\sup\limits_{T-\delta\leq t\leq s\leq T}\mathbb{E}\left[\left.\int^T_s\left|\widetilde{g}_{ZZ}\cdot \Delta D_\theta\overline{Z}(t,\tau)\cdot \overline{Z}^1_t(t,\tau)\right|d\tau\right|\mathcal{F}_s\right]\right\}^2.   
\end{equation*}
Using the H\"{o}lder inequality leads us to estimate
\begin{equation} \label{NuniformE2}
    \begin{split}
        & \sup\limits_{T-\delta\leq t\leq s\leq T}\mathbb{E}\left[\left.\int^T_s\left|\widetilde{g}_{ZZ}\right|^2d\tau\right|\mathcal{F}_s\right]\leq 2\sup\limits_{T-\delta\leq t\leq s\leq T}\mathbb{E}\left[\left.\int^T_s\left|K_\theta(t,\tau,\omega)\cdot\overline{Z}^1(t,\tau)\right|^2d\tau\right|\mathcal{F}_s\right] \\
        & \qquad\qquad
        +\cdots+2\sup\limits_{T-\delta\leq t\leq s\leq T}\mathbb{E}\left[\left.\int^T_s\left|\widetilde{g}_{ZZ}(t,\tau,0,0,0,0)\right|^2d\tau\right|\mathcal{F}_s\right], 
    \end{split} 
\end{equation}
where the omitted three terms correspond to the other three arguments in \( \widetilde{g}_{ZZ} \). Note that we are studying this mapping within a closed ball in \( M[T-\delta,T] \). Therefore, to squeeze \(\delta\)-terms out of \eqref{NuniformE1} and \eqref{NuniformE2} to achieve a contraction, \( K_\theta \) and \( \widetilde{g}_{ZZ} \) need to satisfy (A2') for \( p = 4 \). Furthermore, with the help of (A2''), one can use the arguments similar to Step 3 in Theorem \ref{GeneralWellp} to extend the local solution established under (A2') to any arbitrarily large time horizon. As a side note, (A3) could be excluded from Theorem \ref{ExtendedWellp} as (A2') or (A2'') covers it.
\end{proof}

\begin{remark}[Locally Lipschitz BSVIEs] \label{RemkLocallyL}
    In Step 3 of the proof of Theorem \ref{GeneralWellp}, we showed that the norms of the global solutions are bounded by a constant depending only on the system parameters, which implies that the Lipschitz conditions need only hold within a sufficiently large bounded region in \(\mathbb{R}^{2k} \times \mathbb{R}^{2(k \times n)}\) rather than globally. In other words, the Lipschitz coefficients are non-uniform in both \(\omega \in \Omega\) and \((y, z, \overline{y}, \overline{z})\). Consequently, we can relax the generator conditions to locally Lipschitz (uniform or non-uniform). The BSVIE \eqref{GeneralCase} remains well-posed under these weaker conditions. Moreover, combining (A4) and (A5) with these locally non-uniform Lipschitz conditions ensures continuity of the solution in the \(s\)-direction, which again aids in designing numerical schemes and justifies the intuitive definition of diagonal processes.
\end{remark}

By Theorem \ref{ExtendedWellp}, the stochastic Lipschitz BSVIE \eqref{NonuLBSVIE example} is well-posed over any arbitrarily large time interval. Furthermore, Theorem \ref{ExtendedWellp} allow us to explore mean-variance problems in more stochastic-volatility models within random investment markets; see Section \ref{Sec: TICAPP}.


\section{Markovian BSVIEs and PDEs} \label{sec:MarkovianBSVIEs}
In this section, we are focused on a special class of BSVIEs, namely Markovian BSVIEs, where the randomness in both the generator and the terminal datum solely comes from (the solution of) a forward SDE (FSDE). We extend the well-known Feynman--Kac formula, showing that the unique adapted solution of the FBSVIE \eqref{Markovian FBSVIE} enables us to offer a probabilistic interpretation for the classical solutions to a class of (nonlocal) parabolic PDEs \eqref{nPDE}.

Let us consider the Markovian FBSVIE: 
\begin{equation} \label{Markovian FBSVIE}
    \begin{cases}
        \begin{aligned}
            dX^{t,x}(s) & ~=~ b(s,X^{t,x}(s))ds+\sigma(s)dB(s), \\

            dY^{t,x}(t,s) & ~=~ g(t,s,X^{t,x}(s),Y^{t,x}(t,s),Z^{t,x}(t,s),Y^{t,x}(s,s),Z^{t,x}(s,s))ds \\

            & \qquad +Z^{t,x}(t,s)dB(s), \\

            X^{t,x}(t) & ~=~ x, \quad Y^{t,x}(t,T) ~=~ h(t,X^{t,x}(T)), \quad 0\leq t \leq s\leq T, \quad x\in\mathbb{R}^d,
        \end{aligned} 
    \end{cases}
\end{equation} 
where the coefficients $b:[0,T]\times\mathbb{R}^d\to\mathbb{R}^d$, $\sigma:[0,T]\to\mathbb{R}^{d\times n}$, $g:\Delta[0,T]\times\mathbb{R}^d\times\mathbb{R}^{2k}\times\mathbb{R}^{2(k\times n)}\to\mathbb{R}^k$, and $h:[0,T]\times\mathbb{R}^d\to\mathbb{R}^k$ are all deterministic functions. Compared to the earlier general BSVIE \eqref{GeneralCase}, where the generator and terminal datum depend on the trajectory $\omega\in\Omega$, the randomness of the Markovian BSVIE \eqref{Markovian FBSVIE} is fully captured through \( X^{t,x} \). In the Markovian setting, we impose the following requirements on the deterministic functions $b$, $\sigma$, $g$, and $h$ of \eqref{Markovian FBSVIE}.
\begin{enumerate}[start=0,label={\upshape(\bfseries B\arabic*).}]
    \item $b(s,x)$ is a $C^{1}_b$-function for all $s\in[0,T]$ and $\sigma(s)$ is bounded and satisfies the uniform ellipticity condition: $(\sigma\sigma^\top)(s)\geq \delta I$ for some $\delta>0$. 
    \item $h(t,x)\in C^{1,1}_b([0,T]\times\mathbb{R}^d;\mathbb{R}^k)$. 
    \item $g(t,s,x,y,z,\overline{y},\overline{z})\in C^{1,1,1,3,3,3,3}_b(\Delta[0,T]\times\mathbb{R}^d\times\mathbb{R}^{2k}\times\mathbb{R}^{2(k\times n)};\mathbb{R}^k)$.
\end{enumerate}

As we will show, (B0)–(B2) not only ensure the global well-posedness of \eqref{Markovian FBSVIE} in \(\Delta[0,T]\), but also guarantee that the associated PDE \eqref{nPDE} has a unique, smooth classical solution in \(\Delta[0,T] \times \mathbb{R}^d\), being consistent with the results in the related PDE literature \cite{Wang2021, lei2023nonlocal, LeiSDG, lei2023well}. This provides a soild linkage between PDEs and SDEs. Moreover, by requiring the Malliavin derivative of the terminal datum \(Y^{t,x}(t,T)\) to be bounded, we also ensure the boundedness of \(D_\theta X^{t,x}(T)\), which in turn requires the diffusion term \(\sigma\) of \(X^{t,x}\) to be independent of \(X^{t,x}\), as shown in Lemma \ref{BoundedX} below. 

Next, we define \(\nabla (X^{t,x}, Y^{t,x}, Z^{t,x})(t,s) := (\nabla X^{t,x}(s), \nabla Y^{t,x}(t,s), \nabla Z^{t,x}(t,s))\), which denotes the matrix of first-order partial derivatives of \((X^{t,x}(s), Y^{t,x}(t,s), Z^{t,x}(t,s))\) with respect to the initial condition \(x\) of \(X^{t,x}(s)\). We then have the following result:
\begin{lemma} \label{BoundedX}
    Under (B0), for any $(t,s)\in[0,T]\times\mathbb{R}^d$, the forward SDE of \eqref{Markovian FBSVIE} admits a unique adapted solution $X^{t,x}(s)$ with bounded $D_\theta X^{t,x}(s)$ and $\nabla X^{t,x}(s)$. Moreover, for $0\leq t\leq \theta\leq s\leq T$, 
    \begin{equation} \label{LinkX}
        D_\theta X^{t,x}(s)=\nabla X^{t,x}(s)(\nabla X^{t,x}(\theta))^{-1}\sigma(\theta).
    \end{equation}
\end{lemma}
\begin{proof}
    Under (B0), by the classical SDE theory \cite{Yong1999}, the forward SDE of \eqref{Markovian FBSVIE} admits a unique adapted solution $X^{t,x}(s)$. Moreover, the (random) ODEs satisfied by $\nabla X^{t,x}(s)$ and $D_\theta X^{t,x}(s)$ (given in \eqref{FBSDEx} and \eqref{FBSDEtheta} below) admit a unique solution given by an explicit representation 
    \begin{equation*}
        \begin{cases}
            \begin{aligned}
            \nabla X^{t,x}(s) & ~=~ I\cdot\exp\left\{\int\nolimits^s_tb_X(\tau,X^{t,x}(\tau))d\tau\right\}, \\

            \\[-15pt]

            D_\theta X^{t,x}(s) & ~=~ \sigma(\theta)\cdot\exp\left\{\int\nolimits^s_\theta b_X(\tau,X^{t,x}(\tau))d\tau\right\},
            \end{aligned}
        \end{cases}
    \end{equation*}
    which directly indicates the boundedness of $D_\theta X^{t,x}(s)$ and $\nabla X^{t,x}(s)$, while the \eqref{LinkX} connects between them. 
\end{proof}

With the help of Lemma \ref{BoundedX}, we can show that the well-posedness of solutions of the FBSVIE \eqref{Markovian FBSVIE} and the relationship between $D_\theta (Y^{t,x},Z^{t,x}))$ and $\nabla (Y^{t,x},Z^{t,x}))$.  
\begin{theorem} \label{MarkWellp}
    Under assumptions (B0)-(B2), the Markovian FBSVIE \eqref{Markovian FBSVIE} admits a unique solution $(X^{t,x},Y^{t,x},Z^{t,x})(t,s)$. Moreover, for $(s,\tau)\in\Delta[t,T]$ and $\theta\in[t,\tau]$, 
    \begin{equation} \label{LinkYZ}
    \begin{cases}

        \begin{aligned}
        D_\theta Y^{t,x}(s,\tau) & ~=~ \nabla Y^{t,x}(s,\tau)(\nabla X^{t,x}(\theta))^{-1}\sigma(\theta), \\

        D_\theta Z^{t,x}(s,\tau) & ~=~ \nabla Z^{t,x}(s,\tau)(\nabla X^{t,x}(\theta))^{-1}\sigma(\theta).
        \end{aligned}
    \end{cases}
\end{equation} 
In particular, $\{D_\tau Y^{t,x}(s,\tau)\}_{0\leq t\leq s\leq \tau\leq T}$ provides a version of $\{Z^{t,x}(s,\tau)\}_{0\leq t\leq s\leq \tau\leq T}$, i.e. 
\begin{equation} \label{LinkY}
    Z^{t,x}(s,\tau) ~=~ \nabla Y^{t,x}(s,\tau)(\nabla X^{t,x}(\tau))^{-1}\sigma(\tau),\quad (s,\tau)\in\Delta[t,T], \quad a.s. 
\end{equation} 
which is continuous in $\tau$ almost surely.
\end{theorem}

\begin{proof}
    By Lemma \ref{BoundedX}, we can easily verify that the generator and terminal data of \eqref{Markovian FBSVIE} satisfy the assumptions (A1)-(A3). Consequently, Theorem \ref{GeneralWellp} implies that the Markovian FBSVIE \eqref{Markovian FBSVIE} admits a unique adapted solution \( (X^{t,x}(s), Y^{t,x}(t,s), Z^{t,x}(t,s)) \) in $\Delta[0,T]$. 
    
    Next, let us examine the relationship \eqref{LinkYZ} between the $x$-directional first-order partial derivatives \( \nabla (Y^{t,x}, Z^{t,x})(t,s) \) and the Malliavin derivatives \( D_\theta (Y^{t,x}, Z^{t,x})(t,s) \). Note that the dynamics of $(\nabla X^{t,x},\nabla Y^{t,x},\nabla Z^{t,x})(t,s)$ satisfies 
    \begin{equation} \label{FBSDEx}
    \begin{cases}
        \begin{aligned}
            d\nabla X^{t,x}(s) & ~=~ b_X(s,X^{t,x}(s))\cdot\nabla X^{t,x}(s)ds, \\

            d\nabla Y^{t,x}(t,s) & ~=~ \Big[g_X(t,s,X^{t,x}(s),Y^{t,x}(t,s),Z^{t,x}(t,s),Y^{t,x}(s,s),Z^{t,x}(s,s))\cdot\nabla X^{t,x}(s) \\
            
            & \qquad +g_Y\cdot\nabla Y^{t,x}(t,s)+g_Z\cdot\nabla Z^{t,x}(t,s)+g_{\overline{Y}}\cdot\nabla Y^{t,x}(s,s)+g_{\overline{Z}}\cdot\nabla Z^{t,x}(s,s)\Big]ds \\

            & \qquad +\nabla Z^{t,x}(t,s)dB(s), \\

            \nabla X^{t,x}(t) & ~=~ I, \quad \nabla Y^{t,x}(t,T) ~=~ h_X(t,X^{t,x}(T))\cdot\nabla X^{t,x}(T), \quad 0\leq t \leq s\leq T
        \end{aligned} 
    \end{cases}
\end{equation} 
    

        

        

        


        

        

        
and the process $(D_\theta X^{t,x}, D_\theta Y^{t,x},D_\theta Z^{t,x})(t,s)$ of their Malliavin derivatives is governed by $D_\theta Y^{t,x}(t,s)=D_\theta Z^{t,x}(t,s)=0$ for $0\leq \theta\leq t\leq T$ or $s< \theta\leq T$, and 
\begin{equation} \label{FBSDEtheta}
    \begin{cases}
        \begin{aligned}
        dD_\theta X^{t,x}(s) & ~=~ b_X(s,X^{t,x}(s))\cdot D_\theta X^{t,x}(s)ds, \\

        dD_\theta Y^{t,x}(t,s) & ~=~ \Big[g_X(t,s,X^{t,x}(s),Y^{t,x}(t,s),Z^{t,x}(t,s),Y^{t,x}(s,s),Z^{t,x}(s,s))\cdot D_\theta X^{t,x}(s) \\

        & \qquad +g_Y\cdot D_\theta Y^{t,x}(t,s)+g_Z\cdot D_\theta Z^{t,x}(t,s)+g_{\overline{Y}}\cdot D_\theta Y^{t,x}(s,s)+g_{\overline{Z}}\cdot D_\theta Z^{t,x}(s,s)\Big]ds \\

        & \qquad +D_\theta Z^{t,x}(t,s)dB(s), \\

        D_\theta X^{t,x}(\theta) & ~=~ \sigma(\theta), \quad D_\theta Y^{t,x}(t,T) ~=~ h_X(t,X^{t,x}(T))\cdot D_\theta X^{t,x}(T), \quad 0\leq t \leq \theta \leq s\leq T. 
        \end{aligned}  
    \end{cases}
\end{equation}
One can find that both solutions of the two linear BSVIEs \eqref{FBSDEx} and \eqref{FBSDEtheta} are unique. Indeed, if we consider the BSVIE satisfied by the difference $(\Delta Y^{t,x},\Delta Z^{t,x})$ between two possible solutions of \eqref{FBSDEx} and \eqref{FBSDEtheta}, one obtains 
\begin{equation*}
    \begin{cases}
        \begin{aligned}
            d\Delta Y^{t,x}(t,s) & ~=~ \Big[g_Y\cdot \Delta Y^{t,x}(t,s)+g_Z\cdot \Delta Z^{t,x}(t,s)+g_{\overline{Y}}\cdot \Delta Y^{t,x}(s,s) \\

            & \qquad +g_{\overline{Z}}\cdot \Delta Z^{t,x}(s,s)\Big]ds +\Delta Z^{t,x}(t,s)dB(s), \\

            \Delta Y^{t,x}(t,T) & ~=~ 0,
        \end{aligned}   
    \end{cases}
\end{equation*} 
which is well-posed and admits a unique zero solution $\|(\Delta Y^{t,x},\Delta Z^{t,x})\|^2_{\mathcal{M}[0,T]}=0$ by Theorem \ref{ExtendedWellp}. Consequently, the relationship of \eqref{LinkY} between $D_\theta (Y^{t,x},Z^{t,x}))$ and $\nabla (Y^{t,x},Z^{t,x}))$ is an immediate consequence of the uniqueness of the solution to equations \eqref{FBSDEx}-\eqref{FBSDEtheta} and the representation \eqref{LinkX}. Finally, by noting that the martingale integrand components \(Z^{t,x}\) can be expressed in terms of the trace of the Malliavin derivative of \(Y^{t,x}\), its continuity follows from that of $\nabla Y^{t,x}$, $\nabla X^{t,x}$, and $\sigma$. 
\end{proof}

\begin{remark}
    In the Markovian case, the continuity of \( Z \) at \( \tau \) is straightforward, as it directly follows from the continuity of \( Y \), \( X \), and \( \sigma \) at \( \tau \). This contrasts with the approach used in Theorem \ref{ContinuityZ}, where we applied Kolmogorov's continuity theorem along with conditions (A4)–(A5) to establish continuity in the non-Markovian setting.
\end{remark}

Next, we will show that the adapted solution \((Y^{t,x}(t,s), Z^{t,x}(t,s))\) of \eqref{Markovian FBSVIE} can be expressed in terms of a function \( u(t, s, x) \) and \( X^{t,x}(s) \), where \( u \) satisfies an associated (nonlocal) PDE \eqref{nPDE}. The function \( u \) represents the backward component \( Y^{t,x}(t,s) \) at time \( t \), given the state \( x \), while \( u_x \) links to \( Z \), reflecting the sensitivity of the solution with respect to \( x \). This interpretation partially explains why in the dynamically optimal investment policy \eqref{Investpolicy} in an incomplete market, the intertemporal hedging demand involves the martingale integrand process (to hedge fluctuations in the state variable); see Section \ref{Sec: TICAPP}.

Let us introduce the BSVIE-associated semi-linear PDE of the following form 
\begin{equation} \label{nPDE}
    \begin{cases}
        \begin{aligned}
            u_s(t,s,x) & + \frac{1}{2}\mathrm{tr}\left\{\left(\sigma\sigma^\top\right)(s)\cdot u_{xx}(t,s,x)\right\}+\big\langle b(s,x),u_x(t,s,x)\big\rangle  ~ \\

            & -g\big(t,s,x,u(t,s,x),u_x(t,s,x)\sigma(s),u(s,s,x),u_x(s,s,x)\sigma(s)\big) = 0,  \\

            u(t,T,x) & ~=~ h(t,x), \quad 0\leq t\leq s\leq T,\quad x\in\mathbb{R}^d, 
        \end{aligned}   
    \end{cases}
\end{equation} 
where the mapping (non-homogeneous term) $-g$ could be nonlinear with respect to all its arguments, and both $s$ and $x$ are dynamical variables while $t$ should be considered as an external time parameter. The nonlocality comes from the dependence on the unknown function $u$ and its derivatives evaluated at not only the local point $(t,s,x)$ but also at the diagonal line of the time domain $(s,s,x)$. A more specific and relevant application of \eqref{nPDE} is the equilibrium HJB equation \eqref{HJBSys} in Section \ref{Sec: TICAPP} that characterizes the equilibrium solution to a TIC stochastic control problem. Intuitively, in the absence of diagonal terms in the PDE, classical BSDE theory tells us that the solution to such a PDE is associated with a family of BSDEs parameterized by \( t \). It is interesting to explore cases where diagonal components are present on both the PDE \eqref{nPDE} and SDE \eqref{Markovian FBSVIE} sides.

\begin{theorem}
Under assumptions (B0)-(B2), the nonlocal PDE \eqref{nPDE} admits a unique $C^{1,1,2}$-regular classical solution in $\Delta[0,T]\times\mathbb{R}^d$. Moreover, the unique adapted solution of FBSVIE \eqref{Markovian FBSVIE} admits an explicit representation 
\begin{equation} \label{FKformula}
    \begin{cases}
        \begin{aligned}
            Y^{t,x}(s,\tau) & ~=~ u(s,\tau,X^{t,x}(\tau)), \\

            Z^{t,x}(s,\tau) & ~=~ u_x(s,\tau,X^{t,x}(\tau))\sigma(\tau)
        \end{aligned}  
    \end{cases}
\end{equation}
for $(s,\tau)\in\Delta[t,T]$, which in turn provides a probabilistic interpretation for the solution $u$ and its gredient $u_x$ of \eqref{nPDE}. 
\end{theorem}
\begin{proof}
    The existence and uniqueness of the $C^{1,1,2}$-classical solution to the PDE \eqref{nPDE} in \(\Delta[0,T] \times \mathbb{R}^d\) is shown by earlier related literature \cite{Wang2021, lei2023nonlocal, LeiSDG, lei2023well}. For any fixed $(t,s)\in \Delta[0,T]$, we apply It\^{o}'s formula with the map $\tau\mapsto u(t,\tau,X^{t,x}(\tau))$ on $[s,T]$. First of all, it is clear that $u(t,T,X^{t,x}(T))=g(t,X^{t,x}(T))$. Then, one has  
\begin{equation*}
\begin{split}
     \begin{aligned}
    & \qquad du(t,\tau,X^{t,x}(\tau)) \\
    \\[-25pt]
    & ~=~ \Big[u_s(t,\tau,X^{t,x}(\tau))+\sum^k_{i=1}b_i(\tau,X^{t,x}(\tau))\cdot\frac{\partial u}{\partial x_i}(t,\tau,X^{t,x}(\tau)) \\
    \\[-30pt]
    & \qquad\qquad +\frac{1}{2}\sum^k_{i,j=1}\big(\sigma\sigma\big)_{ij}(\tau)\cdot\frac{\partial^2 u}{\partial x_i\partial x_j}(t,\tau,X^{t,x}(\tau))\Big]d\tau + u_x(t,\tau,X^{t,x}(\tau))\sigma(\tau)dB(\tau) \\
    \\[-20pt]
	& ~=~ g\big(t,\tau,X^{t,x}(\tau),u(t,\tau,X^{t,x}(\tau),u_x(t,\tau,X^{t,x}(\tau))\sigma(\tau), \\
    \\[-15pt]
    & \qquad\qquad 
    u(\tau,\tau,X^{t,x}(\tau),u_x(\tau,\tau,X^{t,x}(\tau))\sigma(\tau)\big)d\tau +u_x(t,\tau,X^{t,x}(\tau))\sigma(\tau) dB(\tau). \\
     \end{aligned}
\end{split}
\end{equation*}
Hence, the triple of processes $(X^{t,x}(\tau),u(s,\tau,X^{t,x}(\tau)),u_x(s,\tau,X^{t,x}(\tau))\sigma(\tau))$ is an adapted solution of the Markovian FBSVIE \eqref{Markovian FBSVIE}. By the uniqueness of the solution to \eqref{Markovian FBSVIE}, we must have \eqref{FKformula}. In particular, $u(t,s,x)=Y^{t,s}(t,s)$. 
\end{proof}

Our well-posedness results for \eqref{Markovian FBSVIE} do not require the generator to be linear in certain arguments, which significantly broadens the applicability of the nonlocal Feynman-Kac formula \eqref{FKformula}. Inspired by \cite{Pardoux1992, Wang2020}, it is of interest to demonstrate the regularity of \(Y\) and set \(u = Y\) to show that it satisfies the nonlocal PDE. The connection \eqref{FKformula} bridges PDEs and SDEs, offering a probabilistic interpretation of solutions to specific PDEs. Furthermore, it is expected to play a key role in developing deep-learning-based solvers for high-dimensional nonlocal PDEs, as demonstrated in \cite{Han2018}.



\section{Time-Inconsistent Stochastic Control Problems} \label{Sec: TICAPP}
In this section, we leverage our BSVIE results to study the continuous-time mean-variance (MV) portfolio selection problem for a sophisticated investor with time-varying risk aversion. As discussed in Section \ref{sec:intro}, finding a dynamically optimal (time-consistent) MV investment policy resembles an intrapersonal game, akin to studies on consumer behavior under time-inconsistent (TIC) preferences; see \cite{Peleg1973} and \cite{Harris2001}. Specifically, the investor views the control policy of her future selves as predetermined and aims to optimally react to them. Consequently, her control policy composes of a pure-strategy Nash equilibrium of this intrapersonal game. 

Specifically, let us consider a TIC stochastic control problem with an objective at time $s$:
\begin{equation} \label{GeneralTIC}
    \begin{split}
        J(s,x;\bm{u}) & =\mathbb{E}_{s,x}\big[F(s,X^{\bm{u}}(T))\big] + G\big(s,\mathbb{E}_{s,x}[\psi(X^{\bm{u}}(T))]\big), \quad 0\leq s\leq T,  
    \end{split}
\end{equation}
where $\mathbb{E}_{s,x}[\cdot]$ is the conditional expectation provided by $X^{\bm{u}}(s)=x$ and the mapping $\bm{u}:[0,T]\times\mathbb{R}^d\to\mathbb{R}^m$ is an admissible control law in the sense that both the underlying dynamics $X^{\bm{u}}$ governed by a forward SDE 
\begin{equation} \label{Statedynamics}
dX^{\bm{u}}(\tau)=b\big(\tau,X^{\bm{u}}(\tau),\bm{u}(\tau,X^{\bm{u}}(\tau))\big)d\tau+\sigma\big(\tau,X^{\bm{u}}(\tau),\bm{u}(\tau,X^{\bm{u}}(\tau))\big)dB(\tau),
\end{equation}
and the TIC objective \eqref{GeneralTIC} are well-defined. Here, $F$, $G$, $b$, and $\sigma$ are all deterministic functions of suitable dimensions. The time inconsistency (TIC) in \eqref{GeneralTIC} arises from two main sources: (1) The cost functional \eqref{GeneralTIC} depends on the initial time reference \( s \) of the sub-problem, meaning that a decision-maker's preferences and decisions may change over time based on factors like remaining time to maturity; (2) The nonlinearity in the conditional expectation arguments of \( G \) violates the BPO.

The fundamental idea behind treating the TIC control problem as an intrapersonal game is to consider a game with a continuum of players (selves) over the interval $[0,T]$. Each player at time $s$ in $[0,T]$ is guided by the TIC objective \eqref{GeneralTIC} and selects an optimal strategy by assuming that future policies will have chosen their own best.
An equilibrium policy of the TIC control problem is then defined as the set of strategies chosen by each player, which remains an equilibrium in any subgame and thus constitutes a subgame-perfect pure-strategy Nash equilibrium. Mathematically, \cite{Bjoerk2014,Bjoerk2017} defined the equilibrium policy and the associated value function as follows.

\begin{definition}[Equilibrium policy and Equilibrium value function] \label{Def:EquiPolicy}
Consider an admissible control law $\widehat{\bm{u}}$ as a candidate
equilibrium law and let $\bm{u}$ be any arbitrary admissible control law and $\Delta s>0$ be any fixed real number. For any initial point $(s,x)\in[0,T]\times\mathbb{R}^d$, define a perturbed policy $\bm{u}_{\Delta s}$ by 
\begin{equation}
    \bm{u}_{\Delta s}(\epsilon,x)={\left\{\begin{array}{l l}{\bm{u}(\epsilon,x)}&{{\mathrm{~~for~~}}s\leq \epsilon< s+\Delta s,~~ x\in\mathbb{R}^{d},} \\
    ~ \\ 
    {\widehat{\bm{u}}(\epsilon,x)}&{{\mathrm{~~for~~}}s+\Delta s\leq \epsilon\leq T,~~ x\in\mathbb{R}^{d}.}\end{array}\right.}
\end{equation}
If the candidate law $\widehat{\bm{u}}$ and the perturbed one $\bm{u}_{\Delta s}$ satisfy the inequality 
\begin{equation} \label{Local optimality} 
			\underset{\Delta s\downarrow 0}{\underline{\lim}}\frac{J(s,x;\widehat{\bm{u}})-J(s,x;\bm{u}_{\Delta s})}{\Delta s}\geq 0
		\end{equation} 
for all $(s,x)\in[0,T]\times\mathbb{R}^d$, then we say $\widehat{\bm{u}}$ is an equilibrium policy and the equilibrium value function is defined by $V(s,x)=J(s,x;\widehat{\bm{u}})$.   
\end{definition}

The equilibrium strategy and value function are identified by a recursive equation, similar to the Bellman equation in stochastic control problems. To understand the TIC recursive equation, it is helpful to first examine it in a discrete setting, as discussed in \cite{Bjoerk2014}. As the time mesh size shrinks, this recursive equation converges to an (equilibrium) HJB equation. Since the continuous-time extension is more illustrative than rigorously formal, the extended HJB system \eqref{HJBSys} is presented as a definition rather than a formal proposition, following \cite{Bjoerk2017}. For a more precise derivation and analysis of the discretization and limiting process, we refer the readers to \cite{Wei2017, Yong2012, He2021, bensoussan2013mean}.


\begin{definition}[Equilibrium HJB system] The extended HJB system of equations for $V(s,x)$, $f(t,s,x)$, and $g(s,x)$ is defined as
follows: for $0\leq t \leq s \leq T$ and $x\in\mathbb{R}^d$, 
\begin{equation} \label{HJBSys}
\left\{
    \begin{array}{l}
        \begin{aligned}
        0 & ~=~ \sup\limits_{a\in\mathcal{U}}\Big\{\mathbb{A}^a V(s,x)-\Delta^a_1-\Delta^a_2 \Big\},  \\

        0 & ~=~ \mathbb{A}^{\widehat{\bm{u}}} f(t,s,x), \\

        0 & ~=~ \mathbb{A}^{\widehat{\bm{u}}} g(s,x), 
        \end{aligned}
    \end{array}
\right. 
\end{equation} 
with the terminal conditions: $V(T,y)=f(T,T,y)+G(T,\psi(y))$, $f(t,T,y)=F(t,y)$, and $g(T,y)=\psi(y)$, where $\widehat{\bm{u}}$ denotes the control law that realizes the supremum in the $V$-equation, $\mathbb{A}^a=\frac{\partial}{\partial s}+b(s,x,a)\frac{\partial}{\partial x}+\frac{1}{2}\sigma^2(s,x,a)\frac{\partial^2}{\partial x^2}$ is the usual controlled infinitesimal operator, and the balancing terms $\Delta^a_1$ and $\Delta^a_2$ are introduced to revive the recursion among the subgame problems indexed by time, given by  
\begin{equation*} 
\left\{
    \begin{array}{l}
        \begin{aligned}
            \Delta^a_1 & ~:=~ \mathbb{A}^a f(s,s,x)-\left.\big(\mathbb{A}^a f(t,s,x)\big)\right|_{t=s},  \\

        \\[-15pt]

        \Delta^a_2 & ~:=~ \mathbb{A}^a G(s,g(s,y))-G_g(s,g(s,x))\cdot\mathbb{A}^a g(s,x), 
        \end{aligned} 
    \end{array}
\right. 
\end{equation*} 
in which $G_g$ is the first-order partial derivative of $G(s,g)$ with respect to $g$. Then, $f$, $g$, and $V$ have the probabilistic representations
\begin{equation} \label{ProbInter}
\left\{
    \begin{array}{l}
        \begin{aligned}
            f(t,s,x) & ~=~ \mathbb{E}\left[F\big(t,X^{\bm{u}}(T)\big)|X^{\bm{u}}(s)=x\right], \\

            g(s,x) & ~=~ \mathbb{E}\big[\psi\big(X^{\bm{u}}(T)\big)|X^{\bm{u}}(s)=x\big], \\

            V(s,x) &~=~ f(s,s,x)+G(s,g(s,x)).
        \end{aligned}
    \end{array}
\right. 
\end{equation} 
Furthermore, the function $\widehat{\bm{u}}$ realizing the supremum in the $V$--equation of \eqref{HJBSys} is an equilibrium policy and $V$ is the corresponding equilibrium value function.  
\end{definition}

Suppose that the supremum term in \eqref{HJBSys} admits a sufficiently regular selection \(\varphi\) that consistently achieves the supremum. By the representation of $V$ in \eqref{ProbInter}, it must depends on functions evaluated at \((s, s, x)\). Substituting it into the remaining two PDEs for \(f\) and \(g\) in \eqref{HJBSys} then yields a (nonlocal) PDE system of the form \eqref{nPDE}, where those unknown functions are evaluated at both \((t, s, x)\) and \((s, s, x)\). For further discussion on the solvability and other properties of the extended HJB system \eqref{HJBSys}, see \cite{Bjoerk2014, Bjoerk2017, Wei2017, lei2023nonlocal, LeiSDG, hernandez2023me}.


\subsection{Dynamic Mean-Variance Portfolio Selection} \label{subsec:MV}
After briefly reviewing the game-theoretic approach to TIC control problems, we now focus on dynamic MV portfolio selection in an incomplete market with stochastic investment opportunities. We first present a probabilistic (BSVIE) representation of dynamically optimal MV portfolios. Then, we prove the existence and uniqueness of the solutions to the BSVIE, ensuring the well-definedness of the equilibrium investment policy. Finally, we explore various stochastic environments, highlighting the practical advantages of our probabilistic approach.

First of all, let us introduce Markowitz's MV objective as follows: 
\begin{equation} \label{MVproblem}
    J(s,r,w;\bm{u})=\omega(s)\mathbb{E}_{s,r,w}[W^{\bm{u}}(T)]-\frac{\gamma}{2}\upsilon(s)\mathrm{Var}_{s,r,w}[W^{\bm{u}}(T)], 
\end{equation}
where $(s,r,w)\in[0,T]\times\mathbb{R}^{l+1}$, $\gamma$ is the risk-aversion coefficient, and the information of $\mathcal{F}_s$ provided at $s$ is the state variable $R(s)=r$ and the investor's wealth state $W(s)=w$. In addition, $\omega(s)\geq 0$ and $\upsilon(s)> 0$ are any first differentiable functions, which indicate the investor's weight between conditional expectation $\mathbb{E}_{s,r,w}[\cdot]$ (reward) and conditional variance $\mathrm{Var}_{s,r,w}[\cdot]$ (risk). This objective \eqref{MVproblem} captures both rewards and risks, while allowing for dynamic adjustments of the investor's aversion to risk based on the timing of decisions. Moreover, the underlying dynamics of $(S,R,W)(\cdot)$ is specified as follows:   
\begin{equation} \label{Incompletedynamics}
\left\{
    \begin{array}{l}
     \begin{aligned}
         \displaystyle{\frac{dS(s)}{S(s)}} & ~=~ \mu(s,R(s))ds+{\sigma}(s,R(s))dB^S(s), \\

        dR(s) & ~=~m(s,R(s))ds+n(s)dB^R(s), \\

        dW(s) & ~=~[r W(s)+\bm{u}(s)\beta(s,R(s))]ds+\bm{u}(s)\sigma(s,R(s))dB^S(s),
     \end{aligned}
    \end{array}
\right. 
\end{equation}
where the investor chooses an investment policy of the dollar amount invested in the stock at time $s$, an adapted process $\bm{u}$. Here, the mean return rate $\mu$ and volatility rate $\sigma$ of the stock price $S$ are deterministic functions depending on temporal variable $s$ and state variable $R(s)$, and we denote $\beta=\mu-r_f$ as the excess return rate over the risk-free rate $r_f$. $m$ and $n$ of the state variable $R(\cdot)$ are also both deterministic functions of suitable dimensions. We then study the problem of our interest on a filtered probability space $(\Omega,\mathcal{F}, \{\mathcal{F}_s\}_{s\in[0,T]}, \mathbb{P})$, where two correlated Brownian motions, $B^S$ and $B^R$ with correlation coefficient $\varrho\in[-1,1]$ are defined. All stochastic processes are assumed to be adapted to $\mathbb{F}:=\{\mathcal{F}_s\}_{s\in[0,T]}$, the augmented filtration generated by $B^S$ and $B^R$. We can observe that in this framework, the market is \textbf{incomplete} because trading in stocks and bonds cannot perfectly hedge against the changes in the stochastic investment opportunity set. However, in special cases where there is perfect correlation between the stock return and the state variable (i.e., \(\rho = \pm 1\)), dynamic market completeness could be achieved. For the case of zero correlation, there is no hedging demand for the state variable, as trading in stocks cannot mitigate the fluctuations in the state variable. These special cases leading to trivial solutions are then excluded from our analyses below.

Before proceeding with further analysis, we first transform equations \eqref{MVproblem} and \eqref{Incompletedynamics} into a more manageable form. Specifically, we rewrite \eqref{MVproblem} to match the structure of \eqref{GeneralTIC}.
\begin{equation} \label{equiMV}
    J(s,r,\widetilde{w};\bm{u})=\mathbb{E}_{s,r,\widetilde{w}}\big[\Phi(s,\widetilde{W}^{\bm{u}}(T))\big]+\Psi\big(\mathbb{E}_{s,r,\widetilde{w}}\big[\widetilde{W}^{\bm{u}}(T)\big]\big), 
\end{equation}
where $\Phi(s,w)=\rho(s)w-\frac{\gamma}{2}w^2$ with $\rho(s)=\frac{\omega(s)}{\upsilon(s)}$, and $\Psi(w)=\frac{\gamma}{2}w^2$. Using It\^{o}'s Lemma, we can rewrite the managed wealth process as follows: 
\begin{equation} \label{Updatedwealth}
    d\widetilde{W}(s)=\bm{u}(s)\widetilde{\beta}(s,R(s))ds+\bm{u}(s)\widetilde{\sigma}(s,R(s))dB^S(s), 
\end{equation}
where $(\widetilde{W},\widetilde{\beta},\widetilde{\sigma})(s,R(s))=(W(s),\beta(s,R(s)),\sigma(s,R(s)))\exp\{r_f(T-s)\}$. It is clear that by appropriately adjusting the system parameters, we can reformulate the MV problem \eqref{equiMV} with \eqref{Updatedwealth} to align with the original problem \eqref{MVproblem}-\eqref{Incompletedynamics}. This transformation simplifies our subsequent analysis and leads to the conclusion that the equilibrium value function and related functions from \eqref{HJBSys} become separable in terms of current wealth. As a result, the equilibrium policy \(\widehat{\bm{u}}\) at \(s\) no longer depends on \(W(s)\), making the MV problem more tractable. Without loss of generality, we note that the two spatial arguments differ by a scaling factor \(\exp\{r_f(T-s)\}\), so we will treat them as equivalent in the analysis, which does not affect the search for the equilibrium investment policy.

Next, by identifying $X=(R,W)$, $\psi(X)=W$, and $B=(B^S,B^R)$ in \eqref{GeneralTIC} and \eqref{Statedynamics}, and noting that the controlled infinitesimal operator $\mathbb{A}^a$ transforms an arbitrary twice continuously differentiable function $\phi(s,r,w)$ as 
\begin{equation*}
    \mathbb{A}^a\phi(s,r,w)=\frac{\partial\phi}{\partial s}+a\widetilde{\beta}\frac{\partial\phi}{\partial w}+m\frac{\partial\phi}{\partial r}+\frac{1}{2}\left(a^2\widetilde{\sigma}^2\frac{\partial^2\phi}{\partial w^2}+n^2\frac{\partial^2\phi}{\partial r^2}+2\varrho a\widetilde{\sigma} n\frac{\partial^2\phi}{\partial w\partial r}\right), 
\end{equation*}
we obtain the extended HJB system \eqref{HJBSys} within the context of \eqref{Incompletedynamics}-\eqref{Updatedwealth} as follows:
\begin{equation} \label{ExampleHJBSys}
\left\{
    \begin{array}{l}
        \begin{aligned}
            0 & ~=~ \sup\limits_{a\in\mathcal{U}}\Big\{\left.\big(\mathbb{A}^a f(t,s,r,w)\big)\right|_{t=s}+\Psi_g\big(g(s,r,w)\big)\cdot\mathbb{A}^a g(s,r,w)\Big\},  \\

            0 & ~=~ \mathbb{A}^{\bm{u}} f(t,s,r,w), \\

            0 & ~=~ \mathbb{A}^{\bm{u}} g(s,r,w),
        \end{aligned}  
    \end{array}
\right. 
\end{equation} 
with the terminal conditions $V(T,r,w) = \rho(T)w$, $f(t,T,r,w)=\Phi(t,w)$, and $g(T,r,w)=w$ for $(t,s,r,w)\in\Delta[0,T]\times\mathbb{R}^{l+1}$, where we note that $\mathbb{A}^aV(s,r,w)=\mathbb{A}^af(s,s,r,w)+\mathbb{A}^a\Psi(g(s,r,w))$. Consequently, the supremum term of \eqref{ExampleHJBSys} reads 
\begin{equation*} 
\begin{array}{l}
        \begin{aligned}
            \sup\limits_{a\in\mathcal{U}}\Big\{& f_s(t,s,r,w)|_{t=s} +\widetilde{\beta} af_w(s,s,r,w)+m f_w(s,s,r,w) \\
            
            & ~+~ \frac{1}{2}\widetilde{\sigma}^2 a^2 f_{ww}(s,s,r,w) +\frac{1}{2}n^2f_{rr}(s,s,r,w)+\varrho n\widetilde{\sigma} a f_{rw}(s,s,r,w) \\

            & ~+~ \gamma g(s,r,w)\Big[g_s(s,r,w)+\widetilde{\beta} a g_w(s,r,w)+m g_w(s,r,w)+\frac{1}{2}\widetilde{\sigma}^2 a^2 g_{ww}(s,r,w) \\

            & ~+~ \frac{1}{2}n^2g_{rr}(s,r,w)+\varrho n\widetilde{\sigma} a g_{rw}(s,r,w)\Big]\Big\}=0,
        \end{aligned}
    \end{array} 
\end{equation*} 
the first order condition of the maximization of which is given by 
\begin{equation*}
    \begin{split}
        & \widehat{\bm{u}}(s,r,w) = -\frac{1}{\sigma^2}\frac{1}{f_{ww}(s,s,r,w)+\gamma g(s,r,w)g_{ww}(s,r,w)}\Big\{\widetilde{\beta}(s,r) f_w(s,s,r,w) \\
        &\qquad 
        +\varrho (n\widetilde{\sigma})(s,r) f_{rw}(s,s,r,w) + \gamma g(s,r,w)\big[\widetilde{\beta}(s,r) g_w(s,r,w)+\varrho (n\widetilde{\sigma})(s,r) g_{rw}(s,r,w)\big]\Big\}.  
    \end{split}
\end{equation*}

Next, following the approach in \cite{Basak2010}, we show that the functions to be determined, \(V\), \(f\), and \(g\), are all separable in \(\widetilde{W}(s)\), and the policy \(\widehat{\bm{u}}\) no longer depends on \(\widetilde{W}(s)\), which make the problem tractable. Using the probabilistic interpretations \eqref{ProbInter} of \(g(s,r,w) = \mathbb{E}_{s,r,w}[\widetilde{W}(T)]\) and \(f(t,s,r,w) = \rho(t)\mathbb{E}_{s,r,w}[\widetilde{W}(T)] - \frac{\gamma}{2}\mathbb{E}_{s,r,w}[(\widetilde{W}(T))^2]\), we have
\begin{equation*}
    \begin{cases}
        \begin{aligned}
            g(s,r,w) & ~=~  w+\mathbb{E}_{s,r,w}\left[\int^T_s\widehat{\bm{u}}(\tau,R(\tau),W(\tau))\widetilde{\beta}(\tau,R(\tau))d\tau\right]=w+\phi(s,r), \\

        f(t,s,r,w) & ~=~ -\frac{\gamma}{2}w^2+\rho(t)w+\rho(t)\mathbb{E}_{s,r,w}\left[\int^T_s\widehat{\bm{u}}(\tau,R(\tau),W(\tau))\widetilde{\beta}(\tau,R(\tau))d\tau\right] \\
         
        & \qquad ~-~
        \gamma\mathbb{E}_{s,r,w}\left[\int^T_s\Big[W(\tau)\widehat{\bm{u}}(\tau,R(\tau),W(\tau))\widetilde{\beta}(\tau,R(\tau))\right. \\

        & \qquad\qquad\qquad\qquad\qquad  
        \left.+\frac{1}{2}\widehat{\bm{u}}^2(\tau,R(\tau),W(\tau))\widetilde{\sigma}^2(\tau,R(\tau))\Big]d\tau\right] \\
         
        & ~=~ 
        -\frac{\gamma}{2}w^2+\varphi(t,s,r)w+\psi(t,s,r).
        \end{aligned} 
    \end{cases}
\end{equation*}
Consequently, one has 
\begin{equation} \label{AnalyticP}
    \begin{split}
        \widehat{\bm{u}}(s,r,w)&=\frac{1}{\gamma\widetilde{\sigma}^2(s,r)}\Big[\widetilde{\beta}(s,r) \varphi(s,s,r)+\varrho (n\widetilde{\sigma})(s,r) \varphi_r(s,s,r) + \gamma \widetilde{\beta}(s,r)\phi(s,r)\Big]. 
    \end{split}
\end{equation}








Thanks to the nonlocal Feynman-Kac formula \eqref{FKformula}, we introduce the following decoupled FBSVIE system, which consists of a forward SDE for \( R(s) \), representing stochastic investment opportunities along with two mutually coupled BSVIEs: \((P(t,s), Q(t,s))\) and \((Y(s), Z(s))\). 
\begin{equation} \label{AppFBSVIE}
\left\{
    \begin{array}{l}
     \begin{aligned}
         dR(s) & ~=~ m(s,R(s))ds+n(s)dB^R(s), \\

        dP(t,s) & ~=~ -\frac{\beta(s,R(s))}{\gamma\sigma^2(s,R(s))}\Big(\beta(s,R(s)) P(s,s) +\varrho (n\sigma)(s,R(s)) Q(s,s) \\

        & \qquad  
        + \gamma\beta(s,R(s)) M(s)\Big)ds+Q(t,s) dB^R(s), \\

        dM(s) & ~=~ -\frac{\beta(s,R(s))}{\gamma\sigma^2(s,R(s))}\Big(\beta(s,R(s)) P(s,s)+\varrho (n\sigma)(s,R(s)) Q(s,s) \\

        & \qquad + \gamma\beta(s,R(s)) M(s)\Big)ds+N(s) dB^R(s), \\

        R(t) & ~=~ r, \quad P(t,T) ~=~ \rho(t), \quad M(T) ~=~ 0, \quad 0\leq t \leq s \leq T, \quad r\in\mathbb{R}^l,
     \end{aligned}  
    \end{array}
\right. 
\end{equation} 
the BSDE \((M, N)(s)\) of which can be viewed as a degenerate BSVIE, where the additional parameter \( t \) takes values only within a singleton. Since the forward SDE is decoupled from the two remaining BSVIE systems and their uncertainties stem from the same Brownian motion \( B^R(\cdot) \), all coefficients in this BSVIE system are known and adapted to the filtration generated by \( B^R(\cdot) \). As a result, our previous well-posedness results are applicable for analyzing its solvability. Based on our previous analysis in Sections \ref{sec:Well}-\ref{sec:MarkovianBSVIEs}, we come to the following conclusion.

\begin{lemma} \label{Applemma}
    If $\rho$ is $C^1_b$ in $t$, $m$ is Lipschitz in $R$, and $\beta$ and $\sigma$ are both $C^1_b$ in $R$, then the FBSVIE \eqref{AppFBSVIE} admits a unique adapted solution $(R,P,Q,M,N)(t,s)$ in $\Delta[0,T]$. 
\end{lemma}

Lemma \ref{Applemma} follows directly from our well-posedness results in Theorem \ref{GeneralWellp} and Theorem \ref{MarkWellp}. Later, we will illustrate with some common stochastic volatility models that satisfy the assumptions made. Moreover, the extension results from Theorem \ref{ExtendedWellp} allow us to relax those assumptions, making the framework more adaptable to various stochastic market environments. Next, with the well-posed FBSVIE \eqref{AppFBSVIE}, we can give a probabilistic representation of the dynamically optimal MV investment policy under a stochastic investment environment.

\begin{proposition} \label{AppThm}
If $\rho$ is $C^1_b$ in $t$, $m$ is Lipschitz in $R$, and $\beta$ and $\sigma$ are both $C^1_b$ in $R$, the dynamically optimal (equilibrium) MV policy of a sophisticated investor is given by
    \begin{equation} \label{Investpolicy}
    \begin{aligned}
        \widehat{\bm{u}}(s,R(s),W(s)) & ~=~ \frac{\beta(s,R(s))}{\gamma\sigma^2(s,R(s))}(P(s,s)+ \gamma M(s))\exp\{r_f(T-s)\} \\ 
        & \quad~
        ~+~\frac{\varrho n(s)}{\gamma\sigma(s,R(s))} Q(s,s)\exp\{r_f(T-s)\}, \quad 0\leq s\leq T,
    \end{aligned}  
\end{equation}
where $(R,P,Q,M,N)(t,s)$ is the unique adapted solution of \eqref{AppFBSVIE}. 
\end{proposition} 

In addition to the independence of the optimal investment policy \eqref{Investpolicy} from the current wealth \( W(s) \), implying that it is free of the randomness of the risky asset process \( B^S \), we make some interesting observations about the optimal investment policy as follows:
\begin{enumerate}
    \item The dynamically optimal investment policy \eqref{Investpolicy} exhibits a familiar structure, consisting of myopic and intertemporal hedging components, as in \cite{Basak2010}. The myopic investment, represented by the first term in \eqref{Investpolicy}, represents the optimal allocation for an investor focused solely on immediate returns while ignoring future optimality. The intertemporal hedging demand, reflected by the second term in \eqref{Investpolicy}, arises as a strategy to mitigate risk from volatile market conditions (stochastic states) aligning with hedging principles in portfolio choice literature since \cite{Merton1971}. In the case of zero correlation (\(\varrho=0\)), no hedging can be accommodated for the state variable, as stock trading does not offset its fluctuations.
    \item The diagonal solution processes \( P(s,s) \) and \( M(s) \) of \eqref{AppFBSVIE} impact only the myopic investment, while the diagonal martingale integrand process \( Q(s,s) \) of \eqref{AppFBSVIE} impact only the hedging demand. When \(\varrho=0\), the FBSVIE \eqref{AppFBSVIE} reduces to a form without \( Q(s,s) \), which is the type of BSVIE commonly studied in the existing literature. However, whenever the correlation coefficient is non-zero, it becomes essential to comprehend and manage the diagonal process \( Q(s,s) \). This paper's in-depth examination of the diagonal martingale integrand process provides crucial support for capturing market fluctuations and addressing the inherent hedging requirements in various stochastic environments.
\end{enumerate}

A natural question about our probabilistic representation \eqref{Investpolicy} through FBSVIEs \eqref{AppFBSVIE} is its necessity over the analytic one \eqref{AnalyticP} for the equilibrium investment policy. We highlight the merits of our approach as follows. (1) Our approach significantly relaxes the underlying assumptions, allowing for more flexible financial market models. A PDE-analytic policy \eqref{AnalyticP} requires the unknown functions to be sufficiently smooth and differentiable to ensure the policy to be well-defined. However, as classical BSDE theory \cite{Pardoux1992,Yong1999} indicates, even when a given PDE admits only a continuous viscosity solution instead of a differentiable classical solution, we can still study the corresponding BSDE. This allows us to appropriately interpret its solution and martingale integrand process, which in turn ensures the well-definedness of the corresponding policy \eqref{Investpolicy}; also see the discussion in \cite{lei2024MV}.
(2) The probabilistic representation based on BSVIEs also facilitates the application of deep-learning-based numerical scheme even with high-dimensional state variable \( R(\cdot)\), which is significant and relevant for modeling and capturing (by multiple factors) the complexities of the financial markets.
In contrast, using a PDE-analytic approach poses challenges in numerical implementation in high dimensions. 
Inspired by \cite{Han2018}, one can leverage the deep-learning capacity to approximate the solution to high-dimensional (nonlocal) PDE, trained with simulated trajectories of the corresponding BSVIEs (regardless of dimensionality), in which the probabilistic representation of the PDE solution is the key ingredient.

At last, we exemplify with a class of stochastic models for the state variable that satisfy the technical assumptions in Lemma \ref{Applemma} and Proposition \ref{AppThm}. We consider the state variable evolves according to a FSDE given by 
\begin{equation} \label{GeneralSVM}
    dF(R(s)) = \left( \theta(s) + \kappa(s) E(R(s)) \right) ds + \sigma(s) \, dB^R(s), 
\end{equation}
where $\theta(s)$ is a deterministic function used to model the drift or average direction of the process, \( \kappa(s) g(R(s)) \) represents a time-dependent adjustment to the drift of \( F(R(s)) \) (potentially to fit certain state-dependent structure, e.g., mean-reverting), and \(\sigma(s)\) is the volatility of the process. The following well-known models are nested within the framework of \eqref{GeneralSVM}: 
\begin{description}
    \item[\textbf{Ho–-Lee model}:] $F(R)=R$ and $\kappa(s)=0$ 
    \item[\textbf{Hull--White model}:] $F(R)=E(R)=R$ and $\kappa(s)<0$. It further nests the \textbf{Ornstein--Uhlenbeck process} and the \textbf{Vasicek model} as well as the \textbf{Brownian bridge}. 
    \item[\textbf{Bessel process}:] $F(R)=R$, $E(R)=1/R$, $\kappa(s)>0$, and $\theta(s)=0$.
\end{description}
By Lemma \ref{Applemma} and Proposition
\ref{AppThm}, choosing suitable \( C^1_b \) functions for \(\beta\) and \(\sigma\) guarantee the well-posedness of the corresponding FBSVIE \eqref{AppFBSVIE}, thereby ensuring a well-defined equilibrium investment policy \eqref{Investpolicy}. Furthermore, in the same spirit of Theorem \ref{ExtendedWellp}, the applicability can be further extended to some cases where \(\beta\) and \(\sigma\) are possibly unbounded.


\section{Conclusion} \label{Sec: Conclusions}
This paper advances the study of BSVIEs by proving the existence and uniqueness of solutions to general BSVIEs with nonlinear dependence on both the solution process and the martingale integrand part, as well as on their diagonal processes, over an arbitrary time horizon. The introduction of Malliavin calculus provides a fresh perspective on dealing with diagonal processes while distinguishes our work from the existing literature. Our study of BSVIEs also offers a probabilistic representation of classical solutions to a class of semi-linear PDEs via a nonlocal Feynman--Kac formula, paving the way for applying deep learning techniques to solving high-dimensional nonlocal PDEs. Our investigation of time-inconsistent stochastic control problems highlights the practical value of our BSVIE results. Our well-posedness results provide theoretical support of equilibrium investment strategies for dynamic mean-variance portfolio selection in stochastic volatility models. Specifically, in incomplete markets, the myopic strategy consists of the diagonal solution processes while the hedging demand is reflected by the martingale integrand component of the BSVIE solution. 

\end{document}